\documentclass[a4paper,11pt]{article}
\usepackage{mathtools}
\usepackage{amsmath}
\usepackage{mathrsfs}
\usepackage{amssymb}
\usepackage{amsthm}
\usepackage{dsfont}
\usepackage{graphicx}
\usepackage{bmpsize}
\usepackage{times}
\usepackage{textcomp}
\usepackage{mathtools}
\usepackage{latexsym, empheq, fancybox}
\usepackage{mathrsfs}
\usepackage{exscale}
\usepackage{enumerate}
\usepackage{booktabs, array}
\usepackage[english]{babel}
\newtheorem{theorem}{Theorem}[section]
\newtheorem{lem}{Lemma}[section]
\newtheorem{defn}{Definition}[section]

\newtheorem{corollary}{Corollary}[section]
\newtheorem{rem}{\bf Remark}[section]
\newtheorem{claim}{Claim}
\newtheorem{prop}{Proposition}[section]
\newtheorem{con}{Condition}[section]
\usepackage{authblk}

\begin{document}
\title{ Double variational principle  for  mean dimensions with sub-additive potentials
       \footnotetext {*Corresponding author}
		\footnotetext {2010 Mathematics Subject Classification: 37B40, 37C45}}

\author{ Yunping Wang,  Ercai Chen*\\
\small    School of Mathematical Sciences and Institute of Mathematics, Nanjing Normal University,\\
\small   Nanjing 210046, Jiangsu, P.R. China\\

\small    yunpingwangj@126.com,
 ecchen@njnu.edu.cn
}

\date{}

\maketitle{}
\begin{abstract}
In this  paper, we introduce mean dimension quantities  with  sub-additive potentials.  We define mean  dimension with sub-additive potentials and mean metric dimension with sub-additive potentials, and establish a double variational principle for sub-additive potentials.
\end{abstract}
\noindent
\textbf{ Keywords:} mean dimension,  rate distortion dimension, sub-additive potentials, variational principle 

\section{Introduction}
\subsection{Backgrounds}
A pair $(\mathcal{X}, T)$ is called a dynamical system if $\mathcal{X}$ is a compact metrizable space with metric $d$ and $T: \mathcal{X}\rightarrow X$ is a homeomorphism. In classic ergodic theory, measure theoretic entropy and topological entropy are important determinants of complexity in dynamical systems. The important relationship between these two quantities is the well-know variational principle.

Topological pressure is a generalization of topological entropy for a dynamical system. The concept was first introduced by Ruelle
\cite{Rue} in 1973 for expansive maps acting on compact metric spaces.
And he set up a variational principle for the topological pressure in the same paper. In \cite{Wal}, Walter generalized theses results to general continuous maps on a compact metric spaces. Given  a continuous map $T: \mathcal{X} \rightarrow \mathcal{X}$ on a compact metric space, the topological pressure of a continuous function $\varphi: \mathcal{X}\rightarrow \mathbb{R}$ is defined by
\begin{align*}
P(\varphi, T)=\lim\limits_{\epsilon\rightarrow 0} \limsup\limits_{n\rightarrow \infty} \dfrac{1}{n} \log \sup\limits_{E} \sum\limits_{x\in E}\exp \sum\limits_{i=0}^{n-1}\varphi(T^{i}x),
\end{align*}
with the supremum taken over all $(n, \epsilon)$-separated sets $E\subset X$. We recall that a set $E\subset X$ is said to be $(n, \epsilon)$-separated if for any $x, y\in E$ with $x\neq y$ there exists $k\in \left\lbrace 0, \cdots, n-1 \right\rbrace $ such that $d(T^{i}x, T^{i}y)>\epsilon$.
Take $\varphi=0$ we recover the notion of the topological entropy $h(T)$ of  the map $T$ given by
\begin{align*}
h(T)=\lim\limits_{\epsilon \rightarrow 0}\limsup\limits_{n\rightarrow \infty} \dfrac{1}{n} \log N(n, \epsilon)
\end{align*}
where $N(n, \epsilon)$ denotes the maximal cardinality of an $(n,\epsilon)$-separated set. The variational principle formulated by Walter can be stated precisely as follows:
\begin{align*}
P(\varphi, T)=\sup\limits_{\mu}\left( h_{\mu}(T)+\int_{X} \varphi d \mu \right),
\end{align*}
with the supremum taken over all $T$-invariant probability measure $\mu$ on $\mathcal{X}$, and $h_{\mu}(T)$ denotes the measure-theoretical entropy of $\mu$.

The theories of   topological pressure, variational principle and equilibrium states play a fundamental role in statistical mechanics, ergodic theory and dynamical systems (see \cite{Bow}, \cite{Kel}, \cite{Rue1}, \cite{WP}).
Since the works of Bowen \cite{Bow79} and Ruelle \cite{Rue2}, the topological pressure has become a basic tool for studying dimension in conformal dynamical systems. In 1984, Pesin and Pitskel \cite{PS32} defined the topological pressure of additive potentials for non-compact subsets of compact metric spaces and proved the variational principle under some supplementary conditions. In 1988, the sub-additive thermodynamic formalism was introduced by Falconer in \cite{Fal} and he proved the variational principle for topological pressure under some Lipschitz conditions and bounded distortion assumption on the sub-additive potentials. In 1996, Barreira \cite{Bar} defined the topological pressure for an arbitrary sequence of continuous functions on a arbitrary subset of compact metric spaces and proved the variational principle under a strong convergence assumption on the potentials which extended the work of Pesin and Pitskel.   Cao,  Feng and  Huang \cite{Cao} introduced the sub-additive topological pressure via separated sets in \cite{Cao} on general compact metric spaces, and obtained the variational principle for sub-additive potentials without any additional assumptions on the sub-additive potentials. For more research on sub-additive topological pressure,  refer to the literatures \cite{ Zhang, LB, Zhao, Yun}. 

Mean  dimension is a conjugacy invariant of  dynamical systems which was first introduced by  Gromov \cite{GRO}.  In 2000, Lindenstrauss and Weiss \cite{LWE} used it to answer  an open question raised by Auslander \cite{AU} that whether every minimal system $(\mathcal{X},T)$ can be imbedded in $[0, 1]^{\mathbb{Z}}$.  It turns out  that   mean dimension  is the right invariant to study  for the problem  of existence of an embedding into $(([0,1]^{D})^{\mathbb Z}, \sigma)$. Mean dimesion can be  applied to solve imbedding problems in dynamical systems (see \cite{GT}, \cite{LT14},\cite{GLT16}).   The metric mean dimension was introduced in \cite{LWE} and they proved that metric mean dimension is  an upper bound of the  mean dimension. It allowed them to establish the relationship between the mean dimension and the topological entropy of dynamical systems, which shows that each system with finite topological entropy has zero mean dimension.  This  invariant enables one to distinguish systems with infinite topological entropy.  In \cite{LT18}, Lindenstrauss and Tsukamoto   established new variational principles connecting rate distortion function to metric mean dimension, which reveals a close relation between mean dimension and rate distortion theory. This was further developed by \cite{LT}. They  injected ergodic-theoretic concepts into mean dimension and developed  a double variational principle between mean dimension  and rate distortion dimension. They  proved the mean dimension is equaled to the rate distortion dimension with respect to two variables (metric and measures).
Recently,  Tsukamoto \cite{MT}  introduced a mean dimension analogue of topological pressure and proved the pressure version of double variational principle which extended the results of \cite{LT}.
The  variational principle formulated by Tsukamoto can be stated  precisely as follows:
\begin{theorem} \label{main2}
	Let $(\mathcal{X}, T)$ be a dynamical system with the marker property  and let $\varphi: \mathcal{X} \rightarrow \mathbb{R}$ be a continuous function. Then
	\begin{align*}
	{\rm mdim}(\mathcal{X}, T, \varphi)&=\min\limits_{d \in \mathcal{D}(\mathcal{X})} \sup\limits_{\mu \in M(\mathcal{X},T)}\left( \overline{{\rm rdim}}(\mathcal{X}, T, d, \varphi, \mu)+ \int_{\mathcal{X}} \varphi d \mu\right) \\&=\min\limits_{d \in \mathcal{D}(\mathcal{X})} \sup\limits_{\mu \in M(\mathcal{X},T)}\left( \underline{{\rm rdim}}(\mathcal{X}, T, d, \varphi, \mu)+ \int_{\mathcal{X}} \varphi d \mu\right)
	\end{align*}
\end{theorem}
\noindent The proof of Theorem \ref{main2} is along the following steps:
\begin{itemize}
	\item [1.] Define metric mean dimension with potential and prove metric mean dimension with potential bounds rate distortion dimension plus function integral. 
	\item [2.] Define mean Hausdorff dimension with potential and construct a invariant measure by Frostman's lemma \cite{How95}.
	\item[3.] Prove the dynamical version of  Pontrjagin-Schnirelmann's theorem \cite{PS32}: for a compact metrizable space $\mathcal{X}$ they can construct a  metric $d$ on it for which the upper metric dimension with potential is equal to the topological dimension with potential.

\end{itemize}

In this paper, we will introduce mean dimension  quantities with sub-additive potential (mean dimension with sub-additive potential, metric mean dimension with sub-additive potential, mean Hausdorff dimension with sub-additive) and apply  Tsukamoto's steps to prove a double variational  principle with  sub-additive potentials.  We  should emphasize here that technical difficulties arising from sub-additive potentials need to overcome. The paper is organized as follows. In Section \ref{mutal}, we introduce mean dimension quantities for sub-additive potentials and recall some basic properties of mutual information. In Section\ref{a}, we prove Theorem \ref{bound} and Proposition \ref{pro}. In Section \ref{b}, we give a proof of Theorem \ref{thmc}. In Section \ref{d},  we give the proof of Theorem \ref{aa}.
\subsection{Statement of the main result }

\begin{defn}
	A dynamical system $(\mathcal{X}, T)$  is said to have the marker property if for any $N>0$, there exists an open set $U\subset \mathcal{X}$ satisfying
	$$\mathcal{X}=\bigcup\limits_{n\in \mathbb{Z}} T^{-n} U, ~~U\cap T^{-n}U=\emptyset~(\forall 1\leq n \leq N).$$
\end{defn}
\begin{defn}
	A sequence  $\mathcal{F}=\left\lbrace \varphi_{n} \right\rbrace_{n=1}^{\infty}$ of functions on $\mathcal{X}$ is called sub-additive if each $\varphi_{n}$ is continuous real-value function on $\mathcal{X}$ such that
	$$\varphi_{n+m}(x)\leq \varphi_{n}(x)+\varphi_{m}(T^{n}x),~\forall x\in \mathcal{X}, m,n\in \mathbb{N}.$$
	
\end{defn}
For a $T$-invariant Borel probability measure $\mu$, denote $$ \mathcal{F}_{*}(\mu)=\lim\limits_{n\rightarrow \infty} \dfrac{1}{n} \int  \varphi_{n} d \mu.$$
The existence of the above limit follows from a sub-additive argument. We call $\mathcal{F}_{*}(\mu)$ the Lyapunov exponent of $\mathcal{F}$ with respect to $\mu$. It also takes a value in $[-\infty, \infty)$.

Let ${\rm var}_{\epsilon}(\varphi, d)=\sup\{ |\varphi(x)-\varphi(y)|,~ d(x,y)<\epsilon\}.$
If $\mathcal{F}=\left\lbrace \varphi_{n} \right\rbrace_{n=1}^{\infty} $ satisfies the following  assumption:
\begin{align*}
\lim\limits_{ \epsilon \rightarrow 0} \limsup\limits_{n\rightarrow \infty}\dfrac{ {\rm var}_{\epsilon}(\varphi_{n}, d_{n})}{n}=0
\end{align*}
then  $\mathcal{F}$ has bounded distortion.

We denote $\mathcal{D}(\mathcal{X})$ and $\mathcal{M}(\mathcal{X}, T)$ the sets of metrics and invariant probability measures on it respectively.
As a main result, we  obtain the following variational principle.
\begin{theorem}\label{main1}
	Assume that $\overline{{\rm mdim_{M}}}(\mathcal{X}, T, d)<\infty$ for all $d\in \mathcal{D}(X)$.
	Let $\mathcal{F}=\left\lbrace  \varphi_{n} \right\rbrace_{n=1}^{\infty}  $ be a sub-additive potential with bounded distortion  and let $(\mathcal{X}, T)$ be a dynamical system  with the maker property. If there exists $K>0$ such that
	$|\varphi_{n+1}(x)-\varphi_{n}(x)|\leq K, ~\forall x\in \mathcal{X}~,n\in \mathbb{N}.$  Then
	
	\begin{align*}
	{\rm{mdim}}(\mathcal{X},T,  \mathcal{F})&=\min_{{\bf d} \in \mathcal{D}(\mathcal{X})}\sup_{\mu \in M(\mathcal{X}, T)}\left( \overline{\rm{rdim}}(\mathcal{X}, T, d, \mu)+ \mathcal{F}_{*}(\mu)\right)
	\\&= \min_{{\bf d} \in \mathcal{D}(\mathcal{X})}\sup_{\mu \in M(\mathcal{X}, T)}\left( \underline{\rm{rdim}}(\mathcal{X}, T, d, \mu)+ \mathcal{F}_{*}(\mu)\right).	
	\end{align*}

\end{theorem}
The Theorem \ref{main1}  can be obtained from the following theorems.

{\bf Step 1}:  prove  mean Hausdorff dimension with sub-additive potentials bounds  mean dimension with sub-additive potentials and show that the rate-distortion dimension is no more than the metric mean  dimension plus the Lyapunov exponent of $\mathcal{F}$.
\begin{theorem}({=Theorem \ref{bound}})
	Let $(\mathcal{X},T)$ be a dynamical system with a metric $d$, then
	$${\rm mdim}_{H}(\mathcal{X},T, d, \mathcal{F})\leq \underline{\rm{mdim}}_{M}(\mathcal{X},T, d, \mathcal{F}).$$ If $\mathcal{F}$ satisfies bounded distortion and there exists $K>0$ such that $|\varphi_{n+1}-\varphi_{n}|<K$ for every $n$, then
	$$ {\rm mdim}(\mathcal{X},T,  \mathcal{F})\leq {\rm mdim}_{H}(\mathcal{X},T, d, \mathcal{F}).$$ 	
\end{theorem}

\begin{prop}({= Proposition\ref{pro}})
	Let $(\mathcal{X}, T)$ be a dynamical system with a metric ${ d}$ and  an invariant probability measure $\mu$. Let $\mathcal{F}=\left\lbrace \varphi_{n} \right\rbrace_{n=1}^{\infty} $  be a sub-additive potential such that $\mathcal{F}_{*}(\mu)\neq -\infty$. Then
	\begin{align*}
	&\overline{\rm{rdim}}(\mathcal{X},T, d, \mu)+\mathcal{F}_{*}(\mu)\leq \overline{\rm{mdim}}_{M}(\mathcal{X},T, d, \mathcal{F}),\\&
	\underline{\rm{rdim}}(\mathcal{X},T, d, \mu)+\mathcal{F}_{*}(\mu)\leq \underline{\rm{mdim}}_{M}(\mathcal{X},T, d, \mathcal{F}).
	\end{align*}	
\end{prop}
{\bf Step 2:} show that the following results by constructing the measure through a version of dynamical Frostman's lemma.
\begin{theorem}(= Theorem\ref{thmc})
	Assume that $\overline{{\rm mdim_{M}}}(\mathcal{X}, T, d)<\infty$ for all $d\in \mathcal{D}(X)$ and there exists $K>0$ such that $|\varphi_{n+1}-\varphi_{n}|<K$ for every $n$.	Under a mild condition on $d$ (called tame growth  of covering numbers)
	$$ {\rm mdim}_{H}(\mathcal{X}, T, d, \mathcal{F}) \leq \sup\limits_{\mu \in  \mathcal{M}(\mathcal{X}, T)}(\underline{{\rm rdim}}(\mathcal{X}, T, d, \mu)+ \mathcal{F}_{*}(\mu)).$$
	
\end{theorem}
\begin{corollary}
	\begin{align*}
	{\rm{mdim}}(\mathcal{X},T,  \mathcal{F})&\leq \sup\limits_{\mu \in  \mathcal{M}(\mathcal{X}, T)}(\underline{{\rm rdim}}(\mathcal{X}, T, d, \mu)+ \mathcal{F}_{*}(\mu)) \\&\leq \sup\limits_{\mu \in  \mathcal{M}(\mathcal{X}, T)}(\overline{{\rm rdim}}(\mathcal{X}, T, d, \mu)+ \mathcal{F}_{*}(\mu))\leq\overline{\rm{mdim}}_{M}(\mathcal{X},T, d, \mathcal{F})
	\end{align*}
	
\end{corollary}
{\bf Step 3}: construct a metric so that the metric mean dimension  is equal to the mean dimension.
\begin{theorem}{($\subset$ Theorem \ref {le})}
	Let $(\mathcal{X}, T)$ be a dynamical system with a sub-additive potential $\mathcal{F}=\left\lbrace  \varphi_{n}\right\rbrace_{n=1}^{\infty}$.  Suppose $(\mathcal{X}, T)$  has the marker property and there exists $K>0$ such that $|\varphi_{n+1}-\varphi_{n}|<K$ for every $n$. Then there exists $d\in \mathcal{D}(\mathcal{X})$ such that
	$${                                                                                \overline{{\rm mdim_{M}}}}(\mathcal{X}, T, d, \mathcal{F} )={\rm mdim}(\mathcal{X}, T, \mathcal{F}).$$
\end{theorem}

\section{Preliminaries}\label{mutal}
\subsection{ Mean dimension quantities for  sub-additive potentials}
In this subsection, we define the mean dimension quantities for sub-additive potentials.
First, we recall { local dimension }\cite{MT}.
Throughout the paper we assume that simplicial complexes are finite (namely, they have only finitely many simplexes).

Let $P$  be a simplicial complex. For $a\in P$ we define the  {\emph local dimension }  $ {\rm dim}_{a} P$ as the maximum of ${\rm dim} ~\Delta $ where $ \Delta \subset P$ is a simplex of $P$ containing $a$.
Let $(\mathcal{X}, d)$ be a compact metric space and $f: \mathcal{X} \rightarrow \mathcal{Y} $ a continuous map into some topological space $\mathcal{Y}$. For $\epsilon >0$ we call the map $f$ an $\epsilon$-embedding if ${\rm diam} f^{-1} y < \epsilon$ for all $y\in \mathcal{Y}$. Let $\varphi: \mathcal{X} \rightarrow \mathbb{R}$ be a continuous function.
We define the $\epsilon${\emph-width dimension with potential} by
\begin{align*}
\rm{Widim}_{\epsilon}(\mathcal{X}, d, \varphi)=\inf\{
\max\limits_{x\in \mathcal{X}}({\rm dim}_{f(x)} P + \varphi(x))| ~&P ~\text {is a simplicial complex and }\\& f: \mathcal{X} \rightarrow P ~\text{is an } \epsilon\text{-embedding}
\}.
\end{align*}
Let $T: \mathcal{X} \rightarrow \mathcal{X}$ be a homeomorphism.  For $N>0$ we define a metric $d_{N}$ by
$$ d_{N}(x,y)= \max\limits_{0\leq n < N } d(T^{n}x, T^{n}y)~~(x, y \in X).$$
We define the {\emph mean topological dimension for sub-additive potentials} by
\begin{align}\label{mdim}
{\rm mdim} (\mathcal{X}, T, \mathcal{F})= \lim\limits_{\epsilon \rightarrow 0} \left( \lim\limits_{N\rightarrow \infty} \dfrac{{\rm Widim}_{\epsilon}(\mathcal{X}, d_{N},  \varphi_{N})}{N}\right) .
\end{align}
The limits exist because the quantity ${\rm Widim}_{\epsilon}(\mathcal{X},d_{N},  \varphi_{N}) $ is subadditive in $N$ and monotone in $\epsilon$.
The value of ${\rm mdim} (\mathcal{X}, T, \varphi)$ is independent of the choice of $d$.  Namely it becomes a topological invariant of $(\mathcal{X},T)$. So we drop $d$ from the notation. When $\varphi=0$, the above (\ref{mdim}) specializes to the standard mean topological dimension: ${\rm mdim}(\mathcal{X}, T, 0)= {\rm mdim} (\mathcal{X}, T).$

The metric mean dimension for sub-additive potentials is defined as follows. Let $(\mathcal{X}, d)$  be a compact metric space with a continuous function $\varphi: \mathcal{X} \rightarrow \mathbb{R}$.
For $\epsilon> 0$, we set
\begin{align*}
\#(\mathcal{X}, d, \varphi, \epsilon)=\inf \{ \sum_{i=1}^{n} (1/\epsilon)^{\sup_{U_{i}}\varphi}\mid  ~ \mathcal{X}= U_{1} \cup \cdots \cup U_{n} ~&\text{ is an open cover with } \\&{\rm diam}~U_{i} < \epsilon ~\text{for all}  ~1\leq i \leq n \}.
\end{align*}
Given a  homeomorphism $T: \mathcal{X} \rightarrow \mathcal{X}$, we set
$$ P(\mathcal{X}, T, d, \mathcal{F}, \epsilon)= \lim\limits_{N\rightarrow \infty} \dfrac{\log \# (\mathcal{X}, d_{N}, \varphi_{n}, \epsilon)}{N}.$$
This limit exists because $\log \#(\mathcal{X} , d_{N}, \varphi_{n}, \epsilon)$ is subadditive in $N$.

We define the  { upper and lower metric mean dimension with sub-additive potentials} by
$$\overline{{\rm mdim_{M}}}(\mathcal{X}, T, d, \mathcal{F})= \limsup\limits_{\epsilon \rightarrow 0} \dfrac{P(\mathcal{X},T,d, \mathcal{F}, \epsilon)}{\log (1/ \epsilon)},$$
$$\underline{{\rm mdim_{M}}}(\mathcal{X}, T, d, \mathcal{F})= \liminf\limits_{\epsilon \rightarrow 0} \dfrac{P(\mathcal{X},T,d, \mathcal{F}, \epsilon)}{\log (1/ \epsilon)}.$$
When the upper and lower limits coincide, we denote the common value by
${\rm{mdim}}_{M}(\mathcal{X},T,d, \mathcal{F}).$

For $\epsilon>0$ and $s\geq \max\limits_{\mathcal{X}}\varphi$, we set
\begin{align*}
H_{\epsilon}^{s}(X,d,\varphi)=\inf\left\lbrace \sum_{i=1}^{\infty}({\rm diam} E_{i})^{s-\sup_{E_i}\varphi}| \mathcal{X}=\bigcup\limits_{ i=1}^{\infty}E_{i}~ \text{with} ~{\rm diam} E_{i}<\epsilon ~\text{for all} ~ i \geq 1 \right\rbrace
\end{align*}
Here we have used the convention that $0^{0}=1$ and $({\rm diam} \emptyset)^{s}=0$ for all $s\geq 0$. Note that this convention implies $H_{\epsilon}^{\max_{ \mathcal{X}}\varphi}(\mathcal{X}, d , \varphi)\geq 1$. We define ${\rm dim}_{H}(\mathcal{X}, { d},\varphi,\epsilon)$ as the supremum of $s\geq \max_{ \mathcal{X}}$ satisfying $H_{\epsilon}^{s}(\mathcal{X}, d, \varphi)\geq 1.$ Given  homeomorphism $T:\mathcal{X}\rightarrow \mathcal{X}$, we define the { mean Hausdorff dimension for sub-additive potentials} by
\begin{align*}
{\rm mdim_{H}}(\mathcal{X}, T, d, \mathcal{F})=\lim\limits_{ \epsilon \rightarrow 0}\left( \limsup\limits_{N\rightarrow \infty} \dfrac{{\rm dim}_{H}(\mathcal{X}, { d_{N}},\varphi_{n},\epsilon)}{N}\right) .
\end{align*}

We can also define the {  lower mean Hausdorff dimension for sub-additive potentials}
${\rm mdim_{H}}(\mathcal{X}, T, d, \mathcal{F})$  by  replacing $\limsup_{N}$ with $\liminf_{N}$ in this definition. But we do not need this concept in the paper.
\subsection{Mutual information}
In this subsection, we recall some basic properties of mutual information. We omit most of the proofs, which can be found in \cite{LT18}\cite{LT}. Throughout this subsection we fix a probability  space $(\Omega, \mathbb{P})$ and assume that all random variables are defined on it.
Let $\mathcal{X}$ and $\mathcal{Y}$ be measurable spaces, and let  $X$ and $Y$ be random variables taking values in $\mathcal{X}$  and $\mathcal{Y}$ respectively. We define their {\bf mutual information} $I(X,Y)$, which estimates the amount of information shared by $X$ and $Y$.

{\bf Case 1:} Suppose $\mathcal{X}$ and $\mathcal{Y}$ are finite sets.  Then we define
\begin{align*}
I(X; Y)= H(X)+ H(Y)-H(X,Y)=H(X)- H(X| Y).
\end{align*}
More explicitly
$$ I(X; Y)=\sum\limits_{x\in X, y\in Y} \mathbb{P}(X=x, Y=y)\log\dfrac{\mathbb{P}(X=x, Y=y)}{\mathbb{P}(X=x)\mathbb{P}(Y=y)}.$$
Here we use the convention that $0\log (0/a)=0$ for all $a\leq 0$.

{\bf Case 2:} In general,  take measurable maps $f: \mathcal{X} \rightarrow A$ and $g: \mathcal{Y} \rightarrow B$ into finite sets $A$ and $B$. Then we can consider $I(f\circ  X; g\circ Y)$ defined  by Case 1. We define $ I(X; Y)$  as the supremum of  $I(f\circ X; g\circ Y)$ over all finite-range measurable maps $f$ and $g$ defined on $\mathcal{X}$ and $\mathcal{Y}$. This definition is compatible with Case 1 when $\mathcal{X}$ and $\mathcal{Y}$ are finite sets.
\begin{lem}[Date-Processing inequality]\label{dp}
	Let $X$ and $Y$ be random variables taking values in measurable spaces $ \mathcal{X}$ and $\mathcal{Y}$ respectively. If $f:\mathcal{Y} \rightarrow \mathcal{Z}  $ is a measurable map then $I(X; f(Y)) \leq  I(X; Y)$. 	
\end{lem}

\begin{rem}\label{key2}
	Lemma \ref{dp} implies that, in the definition  of  the rate distortion function $R_{\mu}(\epsilon)$, we can assume that the random variable $Y$ there takes only finitely many values, namely that its  distribution  is supported on a finite set. 
\end{rem}

\begin{lem}\label{le1}
	Let $\mathcal{X}$ and $\mathcal{Y}$  be finite sets and let $(X_{n}, Y_{n})$ be a sequence of random variables taking values in $\mathcal{X} \times \mathcal{Y}$. If $(X_{n}, Y_{n})$ converges to some $(X, Y)$ in law, then $I(X_{n}; Y_{n})$ converges to $I(X; Y)$.
\end{lem}

\begin{lem}[Subadditivity of mutual information]\label{lems} Let $X, Y, Z$ be random variables  taking values in finite sets $\mathcal{X}, \mathcal{Y}, \mathcal{Z}$ respectively. Suppose $X$ and $Y$ are conditionally independent given $Z$. Namely for every $z \in \mathcal{Z}$  with $\mathbb{P}(Z=z)\neq 0$
	$$ \mathbb{P}(X=x, Y=y| Z=z)= \mathbb{P}(X=x| Z=z) \mathbb{P}(Y=y| Z=z).$$
	Then $I(X, Y ; Z)\leq I(X; Z)+ I(Y; Z).$
\end{lem}

Let  $X$ and $Y$ be random variables taking values in finite sets $\mathcal{X}$ and $\mathcal{Y}$. We set $\mu(x)= \mathbb{P}(X=x)$ and $\nu(y | x)= \mathbb{P}(Y=y | X=x)$, where the latter is defined only for $x\in \mathcal{X}$ with $\mathbb{P}(X=x)\neq 0$. The mutual information $I(X ; Y)$  is determined  by the distribution of $(X, Y)$, namely $\mu(x)\nu(y| x)$. So we sometimes write $I(X; Y)= I(\mu, \nu).$

\begin{lem}\label{lemc}[Concavity / convexity of mutual information] In this notation, $I(\mu, \nu)$ is a concave function of $\mu(x)$ and a convex function of $\nu(y| x)$. Namely for $ 0\leq t \leq 1$
	$$I((1-t)\mu_{1}+ t \mu_{2}, \nu)\geq (1-t)I(\mu_{1}, \nu)+ t I(\mu_{2}, \nu),$$
	$$I(\mu, (1-t)\nu_{1}+ t \nu_{2}) \leq  (1-t)I(\mu, \nu_{1}) + t I(\mu, \nu_{2}).$$		
\end{lem}
\begin{lem}[Superadditivity  of mutual information]
	Let $X, Y, Z$ be measurable maps from $\Omega$ to $\mathcal{X}, \mathcal{Y}, \mathcal{Z}$ respectively. Suppose $X$ and $Z$ are independent. Then
	$$ I(Y; X,Z)\geq I(Y; X)+ I(Y; Z).$$	
\end{lem}
The following lemma is a key  to connect geometric measure theory  to rate distortion theory\cite{KD94}\cite{LT}.
\begin{lem}\label{ml2}
	Let $\epsilon$ and $\delta$ be positive numbers with  $2\epsilon \log (1/\epsilon)\leq \delta$. Let $0\leq  \tau \leq \min(\epsilon/3, \delta/ 2)$ and $s\geq 0$. Let $(\mathcal{X}, d)$ be a compact metric space with a Borel probability measure $\mu$ satisfying
	\begin{align}
	\mu(E)\leq (\tau + {\rm diam}E)^{s}, ~~~\forall E\subset \mathcal{X} ~\text{with}~{
		\rm diam}E<\delta.
	\end{align}
	Let $X$ and $Y$ be random variables taking values in $\mathcal{X}$ with {\rm Law}(X)=$\mu$ and $\mathbb{E}d(X,Y)<\epsilon.$ Then $$ I(X; Y)\geq s \log (1/\epsilon)- T(s+1).$$
	Here $T$ is a universal positive constant independent of $\epsilon$, $\delta, \tau, s, (\mathcal{X}, d), \mu$.
\end{lem}
\subsection{Rate distortion function}
In this subsection, we briefly review rate distortion theory here. Its primary object is data compression  of continuous random variables and their process. Continuous random variables always have infinite entropy, so it is impossible to describe them perfectly with only finitely many bits. Instead rate distortion theory studies a  lossy data compression method achieving some distortion constrains.
For a couple $(X, Y)$ of random variables we denote its mutual information by $I(X,Y)$. Let $(\mathcal{X}, T)$  be a dynamical system with a distance $d$ on $\mathcal{X}$. Take an invariant probability $\mu\in M(\mathcal{X}, T)$. For  a positive number $\epsilon$ we define the rate distortion function $R_{\mu}(\epsilon)$ as the infimum of
\begin{align}\label{mu}
\dfrac{I(X,Y)}{n},
\end{align}
where $n$ runs over all natural numbers, and $X$ and $Y=(Y_{0}, \cdots, Y_{n-1})$ are random variables defined on  some probability space $(\Omega, \mathbb{P})$ such that
\begin{itemize}
	\item $X$ takes values in $\mathcal{X}$  and its law is given by $\mu$.
	\item Each $Y_{k}$ takes values in $\mathcal{X}$ and $Y$ approximates the process $(X, TX, \cdots, T^{n-1}X)$ in the sense that
	\begin{align}\label{co1}
	\mathbb{E}\left( \dfrac{1}{n}\sum\limits_{k=0}^{n-1}d(T^{k}X, Y_{k})\right) < \epsilon.
	\end{align}
	
\end{itemize}
Here $\mathbb{E}$ is the expectation with respect to the probability measure $\mathbb{P}$. Note that $R_{\mu}(\epsilon)$ depends on the distance $d$ although it is not explicitly  written in the notation.

We define the { upper and  lower rate distortion dimension} by
\begin{align*}
&\overline{\rm rdim}(\mathcal{X}, T, d, \mu)=\limsup\limits_{\epsilon\rightarrow 0} \dfrac{R_{\mu}(\epsilon)}{\log (1/ \epsilon)},\\ &
\underline{\rm rdim}(\mathcal{X}, T, d, \mu)=\liminf\limits_{\epsilon\rightarrow 0} \dfrac{R_{\mu}(\epsilon)}{\log (1/ \epsilon)}.
\end{align*}
When the upper and lower limits coincide,  we denote their common value ${\rm rdim}(\mathcal{X}, T, d, \mu)$.

\section{Mean Hausdorff dimension with sub-additive potentials bounds  mean dimension with sub-additive potentials }\label{a}
In this section, we prove Theorem \ref{bound} and Proposition \ref{pro}. The main issue is to prove that Hausdorff dimension with sub-additive potentials bounds mean dimension with sub-additive potentials.
\subsection{Proof of Proposition \ref{pro} }
\begin{lem}\cite{WP}\label{wp}
	Let $a_{1}, \cdots, a_{n}$ be real numbers and ${\bf p}=(p_{1}, \cdots, p_{n})$ a probability vector. For $\epsilon >0$
	$$ \sum\limits_{i=1}^{n} (-p_{i} \log p_{i}+ p_{i} a_{i} \log(1/ \epsilon)) \leq \log (\sum\limits_{i=1}^{n}(1/ \epsilon)^{a_{i}})$$ and equality holds iff
	$$p_{i}= \dfrac{(1/\epsilon)^{a_{i}}}{\sum_{j=1}^{n} (1/\epsilon)^{a_{j}}}. $$
\end{lem}
\begin{prop}{\label{pro}}
	Let $(\mathcal{X}, T)$ be a dynamical system with a metric ${ d}$ and  an invariant probability measure $\mu$. Let $\mathcal{F}=\left\lbrace \varphi_{n} \right\rbrace_{n=1}^{\infty} $  be a sub-additive potential such that $\mathcal{F}_{*}(\mu)\neq -\infty$. Then
	\begin{align*}
	&\overline{\rm{rdim}}(\mathcal{X},T, d, \mu)+\mathcal{F}_{*}(\mu)\leq \overline{\rm{mdim}}_{M}(\mathcal{X},T, d, \mathcal{F}),\\&
	\underline{\rm{rdim}}(\mathcal{X},T, d, \mu)+\mathcal{F}_{*}(\mu)\leq \underline{\rm{mdim}}_{M}(\mathcal{X},T, d, \mathcal{F}).
	\end{align*}		
\end{prop}
\begin{proof}
	Let $X$ be a random variable taking values in $\mathcal{X}$ and obeying $\mu$. Let $N>0$ and let $\mathcal{X}=U_{1}\cup\cdots\cup U_{n}$ be an open cover with diam($U_{i}, d_{N})< \epsilon$ for all $i$. Pick $x_{i} \in U_{i}$. We define a random variable $Y$ by
	$$Y=(x_{i}, Tx_{i}, \cdots, T^{N-1}x_{i})  ~~~\text{if} ~~X\in U_{i} \setminus (U_{1}\cup\cdots\cup U_{i-1})$$
	Obviously
	$$ \dfrac{1}{N}\sum\limits_{ k=0}^{N-1}\mathbb{E} d(T^{k}X, Y_{k})<\epsilon.$$
	Set $p_{i}=\mu(U_{i}\setminus (U_{1}\cup\cdots \cup U_{i-1})).$ Then
	$$ I(X;Y)\leq H(Y) \leq -\sum\limits_{ i=1}^{n} p_{i}\log p_{i}.$$
	Set $a_{i}=\sup_{U_{i}}  \varphi_{N}.$ It follows that
	\begin{align*}
	R(d, \mu, \epsilon)+(\dfrac{1}{N}\int_{\mathcal{X}}  \varphi_{N} d \mu) \log 1/\epsilon&\leq \dfrac{I(X;Y)}{N}+ (\dfrac{1}{N}\int_{\mathcal{X}} \varphi_{N} d \mu)\log 1/\epsilon\\&\leq \dfrac{1}{N}\sum\limits_{ i=1}^{n}(-p_{i}\log p_{i}+p_{i}a_{i}\log(1/\epsilon))\\&\leq\dfrac{1}{N}\log(\sum\limits_{ i=1}^{n}(\log(1/\epsilon)^{a_{i}}) ~~~\text{by Lemma \ref{wp}.}
	\end{align*}
	Hence
	$$ R(d,\mu, \epsilon)+(\dfrac{1}{N}\int_{\mathcal{X}} \varphi_{N} d \mu) \log {1}/{\epsilon}\leq \dfrac{\log\#(\mathcal{X}, d_{N},  \varphi_{N},\epsilon)}{N}.$$
	Let $N\rightarrow \infty$. Then
	$$ R(d, \mu, \epsilon)+\mathcal{F}_{*}(\mu)\log{1}/{\epsilon}\leq P(\mathcal{X}, T, d, \mathcal{F}, \epsilon).$$
	Divide this by $\log(1/\epsilon)$ and take the limit of $\epsilon\rightarrow 0$.
\end{proof}

\subsection{Proof of Theorem \ref{bound}}
In order to prove Theorem \ref{bound}, we need to give an additional issue around the quantity ${\rm Widim}_{\epsilon}(\mathcal{X}, d, \varphi)$.
Let $P$ be a simplicial complex and $a\in P$.
Recall that  {\bf small local dimension}(\cite{LT}).
$${\rm dim}_{a}^{'}P =\min \left\lbrace  {\rm dim} \Delta : \Delta\subset P ~\text{is a simplex containing a} \right\rbrace .$$
The local dimension ${\rm dim}_{a}P$ is a topological quantity. However, the small local dimension ${\rm dim}_{a}^{'}P$ is a combinatorial quantity. It depends on  the combinatorial structure of $P$.
In \cite{LT}, authors introduced the other definition $\epsilon$-width dimension with potential $ \rm{Widim}_{\epsilon}^{'}(\mathcal{X}, d, \varphi) $  by small local dimension and showed  the following result.
\begin{lem}\label{ieq1}\cite{LT}
	\begin{align*}
	{\rm Widim}_{\epsilon}^{'}(\mathcal{X}, d, \varphi)\leq {\rm Widim}_{\epsilon}(\mathcal{X}, d, \varphi)\leq {\rm Widim}_{\epsilon}^{'}(\mathcal{X}, d, \varphi)+{\rm var}_{\epsilon}(\varphi,d)
	\end{align*}
	where ${\rm var}_{\epsilon}(\varphi,d)=\sup\left\lbrace \left|  \varphi(x)-\varphi(y)\right|d(x,y)\leq\epsilon \right\rbrace .$
\end{lem}
If we put some bound distortion assumption on $\mathcal{F}$, the we can also show the equivalence of these two quantities.

\begin{prop}\label{proa}
	Assume that $\mathcal{F}=\left\lbrace \varphi_{n} \right\rbrace_{n=1}^{\infty} $ satisfies bounded distortion.
	Then we have
	$$ {\rm{mdim}}(\mathcal{X},T,  \mathcal{F})=\lim\limits_{ \epsilon \rightarrow 0}\left( \lim\limits_{N\rightarrow\infty} \dfrac{{\rm Widim}_{\epsilon}^{'}(\mathcal{X}, d_{N}, \varphi_{N})}{N}\right) .$$ Here ${{  \rm Widim}_{\epsilon}^{'}(\mathcal{X}, d_{N}, \varphi_{N})}$ is subadditive in $N$ and monotone in $\epsilon$.
\end{prop}
\begin{proof}
	Recall that we defined
	$${\rm{mdim}}(\mathcal{X},T,  \mathcal{F})=\lim\limits_{ \epsilon \rightarrow 0}\left( \lim\limits_{N\rightarrow\infty} \dfrac{{\rm Widim}_{\epsilon}(\mathcal{X}, d_{N}, \varphi_{N})}{N}\right) . $$
	From Lemma \ref{ieq1} and bound distortion, we have
	$${\rm Widim}_{\epsilon}^{'}(\mathcal{X}, d_{N},  \varphi_{N})\leq {\rm Widim}_{\epsilon}(\mathcal{X}, d_{N}, \varphi_{N})\leq {\rm Widim}_{\epsilon}^{'}(\mathcal{X}, d_{N},  \varphi_{N})+{\rm var}_{\epsilon}(\varphi_{N},d_{N}). $$
	By Proposition \ref{ieq1}, we can get the result.	
\end{proof}
\begin{theorem}\label{bound}
	Let $(\mathcal{X},T)$ be a dynamical system with a metric $d$, then
	$${\rm mdim}_{H}(\mathcal{X},T, d, \mathcal{F})\leq \underline{\rm{mdim}}_{M}(\mathcal{X},T, d, \mathcal{F}).$$ If $\mathcal{F}$ satisfies bounded distortion and there exists $K>0$ such that $|\varphi_{n+1}-\varphi_{n}|<K$ for every $n$, then
	$$ {\rm mdim}(\mathcal{X},T,  \mathcal{F})\leq {\rm mdim}_{H}(\mathcal{X},T, d, \mathcal{F}).$$ 	
\end{theorem}
\begin{proof}
	We firstly show that ${\rm mdim}_{H}(\mathcal{X},T, d, \mathcal{F})\leq \underline{\rm{mdim}}_{M}(\mathcal{X},T, d, \mathcal{F}).$ Let $0<\epsilon<1$ and $N>0$. Let $\mathcal{X}=U_{1}\cup\cdots\cup U_{n}$ be an open cover with ${\rm diam}(U_{i}, d_{N})<\epsilon$. For $s\geq \max_{\mathcal{X}}  \varphi_{N}$
	
	\begin{align*}
	H_{\epsilon}^{s}(\mathcal{X}, d_{N},  \varphi_{N})&\leq \sum\limits_{i=1}^{n}({\rm diam}(U_{i}, d_{N}))^{s-\sup_{U_{i}} \varphi_{N}}\\&\leq\sum_{i=1}^{n}\epsilon^{s-\sup_{U_{i}} \varphi_{N}}=\epsilon^{s}\cdot \sum\limits_{ i=1}^{n}(1/\epsilon)^{\sup_{U_{i}}  \varphi_{N}}.
	\end{align*}	
	Hence
	$$ H_{\epsilon}^{s}(\mathcal{X}, d_{N},  \varphi_{N})\leq \epsilon^{s} \cdot \#(\mathcal{X}, d_{N}, \varphi_{N}, \epsilon).$$
	This implies
	$${\rm dim}_{H}(\mathcal{X}, d_{N}, \varphi_{N}, \epsilon)\leq \dfrac{\log \#(\mathcal{X},d_{N},  \varphi_{N}, \epsilon)}{\log(1/\epsilon)}.$$
	Divide this by $N$ and take the limits of $N\rightarrow \infty$:
	$$\limsup\limits_{N\rightarrow \infty} \dfrac{{\rm dim}_{H}(\mathcal{X}, { d_{N}},\varphi_{N},\epsilon)}{N}\leq \dfrac{P(\mathcal{X}, T, d, \mathcal{F}, \epsilon)}{\log (1/\epsilon)}.$$
	Letting $\epsilon \rightarrow 0$, we get ${{\rm mdim}_{H}}(\mathcal{X}, T, d, \mathcal{F})\leq \underline{{\rm mdim_{M}}}(\mathcal{X}, T, d, \mathcal{F}).$
\end{proof}
Next we show that mean Hausdorff dimension with sub-additive potentials bounds mean dimension with sub-additive potentials. We need some lemmas.
Let $(\mathcal{X}, d)$ be a compact metric space. For  $s\geq 0$, we define
\begin{align*}
H_{\infty}^{s}(\mathcal{X},d)= \inf\left\lbrace \sum\limits_{ i=1}^{\infty} ({\rm diam} E_{i})^{s}| \mathcal{X}=\bigcup\limits_{ i=1}^{\infty} E_{i} \right\rbrace .
\end{align*}
We denote the standard Lebesgue measure on $\mathbb{R}^{N}$ by $\nu_{N}$. We set $\left\|  x\right\| =\max\limits_{ 1\leq i \leq N} \left| x_{i}\right|   $ for  $x\in \mathbb{R}^{N}$. For $A\subset \left\lbrace 1,2,\cdots,N \right\rbrace $ we define $\pi_{A}: [0, 1]^{N} \rightarrow [0,1]^{A}$ as the projection to the A-coordinates. The next Lemma was given in \cite{LT}.
\begin{lem}
	Let $K\subset [0,1]^{N}$ be a closed subset  and $0\leq n \leq N $,
	\begin{itemize}
		\item  $\nu_{N}(K) \leq 2^{N} H_{\infty}^{N}(K, \left\| .\right\| ).$
		\item $\nu_{N}(\bigcup\limits_{ |A|\geq n} \pi_{A}^{-1}(\pi_{A}K)) \leq  4^{N} H_{\infty}^{n}(K, \left\|\cdot \right\| ).$
	\end{itemize}
\end{lem}
The following lemma is the key  ingredient of the  proof of Theorem \ref{bound}.
\begin{lem}\label{lema}\cite{MT}
	Let $(\mathcal{X}, d)$ be a compact metric space with a continuous function $\varphi: \mathcal{X} \rightarrow  \mathbb{R}$. Let $\epsilon>0$, $L>0$ and $s\geq \max_{ \mathcal{X}} \varphi$ be real numbers. Suppose there exists a Lipschitz map $f: \mathcal{X}\rightarrow  [0,1]^{N}$ such that
	\begin{itemize}
		\item$ \left\| f(x)-f(y)\right\| \leq L \cdot d(x,y), $
		\item  $\left\| f(x)-f(y)\right\|=1$ if $d(x,y)\geq \epsilon.$
	\end{itemize}
	Moreover, suppose
	$$ 4^{N}(L+1)^{1+s+\left\| \varphi\right\|_{\infty} } H_{1}^{s}(\mathcal{X}, d, \varphi)<1,$$
	where $\left\| \varphi\right\|_{\infty}=\max_{ \mathcal{X}}\left| \varphi \right|.  $ Then
	$$ {\rm Widim}_{\epsilon}^{'}(\mathcal{X}, d, \varphi)\leq s+1.$$
\end{lem}
\begin{proof}[Proof of Theorem \ref{bound}]
	It is sufficient to show that ${\rm mdim}(\mathcal{X}, T, d, \mathcal{F})\leq {\rm mdim}_{H}(\mathcal{X}, T, d, \mathcal{F}).$ Given $\epsilon>0$, we take a Lipschitz map $f: \mathcal{X} \rightarrow [0,1]^{M}$ such that
	$$ d(x,y)\geq \epsilon \Rightarrow \left\| f(x)-f(y)\right\| =1. $$
	Let $L>0$ be Lipschitz constant of $f$, i.e., $\left\| f(x)- f(y)\right\|\leq L\cdot d(x,y). $ For $N>0$ we define $f_{N}: \mathcal{X} \rightarrow [0,1]^{MN}$ by
	$$f_{N}(x)=(f(x), f(Tx), \cdots, f(T^{N-1}x)).$$	
	Then
	\begin{itemize}
		\item $\left\| f_{N}(x)-f_{N}(y)\right\|\leq L\cdot d_{N}(x,y), $
		\item $\left\|f_{N}(x)-f_{N}(y)=1 \right\| $ if $d_{N}(x,y)\geq \epsilon.$
	\end{itemize}
	Put $s > {\rm mdim}_{H}(\mathcal{X}, T, d, \mathcal{F})$. Let $\tau>0$ be arbitrary.  Take  $0<\delta<1$ such that
	\begin{align}\label{c}
	4^{M}\cdot (L+1)^{1+s+\tau+K+ \left\|  \varphi_{1}\right\|_{\infty}  }\cdot \delta^{\tau}<1.
	\end{align}
	Since ${\rm mdim}_{H}(\mathcal{X}, T, d, \mathcal{F})< s$, we can take $0< N_{1} < N_{2}<  N_{3}<\cdots\rightarrow \infty$ satisfying
	${\rm dim}_{H}(\mathcal{X}, d_{N_{i}}, \varphi_{N_{i}}, \delta)< sN_{i}$. Then $H_{\delta}^{sN_{i}}(\mathcal{X},d_{N_{i}}, \varphi_{N_{i}} )<1$ and hence
	$$ H_{\delta}^{(s+\tau)N_{i}}(\mathcal{X}, d_{N_{i}},  \varphi_{N_{i}})\leq \delta^{\tau N_{i}}H_{\delta}^{sN_{i}}(\mathcal{X}, d_{N_{i}}, \varphi_{N_{i}})<\delta^{\tau N_{i}}.$$
	By $(\ref{c})$, we can
	\begin{align*}
	4^{MN_{i}}(L+1)^{1+(s+\tau)N_{i}+\left\| \varphi_{N_{i}}\right\|_{\infty}}H_{1}^{(s+\tau)N_{i}}(\mathcal{X}, d_{N_{i}}, \varphi_{N_{i}})&< \left\lbrace 4^{M}\cdot (L+1)^{1+s+\tau+K+\left\|  \varphi_{1} \right\|_{\infty}}\cdot \delta^{\tau}  \right\rbrace^{N_{i}}\\&<1.
	\end{align*}
	According to Lemma \ref{lema}, we can have
	$$ {\rm Widim}^{'}_{\epsilon}(\mathcal{X}, d_{N_{i}},   \varphi_{N_{i}} )\leq (s+\tau)N_{i}+1.$$
	Hence
	$$\lim\limits_{N\rightarrow \infty}\dfrac{{\rm Widim}^{'}_{\epsilon}(\mathcal{X}, d_{N},   \varphi_{N} )}{N}\leq s+\tau. $$
	Let $s\rightarrow {\rm mdim}_{H}(\mathcal{X}, T, d, \mathcal{F})$, $\tau \rightarrow 0$ and $ \epsilon\rightarrow 0$:
	$$ \lim\limits_{\epsilon\rightarrow 0}\left( \lim\limits_{N\rightarrow \infty} \dfrac{{\rm Widim}^{'}_{\epsilon}(\mathcal{X}, d_{N},  \varphi_{N} )}{N}\right) \leq {\rm mdim}_{H}(\mathcal{X}, T, d, \mathcal{F}).$$
	By Proposition \ref{proa}, this proves ${\rm mdim}(\mathcal{X}, T, d, \mathcal{F}) \leq {\rm mdim}_{H}(\mathcal{X},T, d, \mathcal{F}).$
	\begin{rem}
		The above proof actually shows ${\rm mdim}(\mathcal{X}, T, d, \mathcal{F}) \leq \underline{{\rm mdim}}_{H}(\mathcal{X},T, d, \mathcal{F})$.  It is worth pointing out that the bound distortion is used in the proof of Proposition \ref{proa}.
	\end{rem}	
\end{proof}
\section{Proof of Theorem \ref{thmc}}\label{b}
In this section, we give a proof of Theorem \ref{thmc}. It states that we can construct invariant probability measures capturing dynamical complexity of $(\mathcal{X}, T, d, \mathcal{F})$. We firstly give some notations and lemmas  which are needed in our proof of Theorem \ref{thmc}.
\begin{defn}
	The compact metric space $(\mathcal{X}, d)$ is said to  have { tame growth of covering numbers} if for every $\delta >0$ it holds that
	$$ \lim\limits_{\epsilon \rightarrow 0} \epsilon^{\delta} \log \# (\mathcal{X}, d, \epsilon)=0.$$
\end{defn}
The following result \cite{LT} shows that the tame growth of covering numbers is a fairly mild condition.
\begin{lem}
	Let $(\mathcal{X},d)$ be a compact metric space.  There exists  a metric $d'$  on $\mathcal{X}$ (compatible with the topology)  such that $d'(x,y) \leq d(x,y)$ and that $(\mathcal{X}, d')$ has the tame growth of covering numbers. In particular every compact metrizable space admits a metric having the tame growth of covering numbers.	
\end{lem}
Let $(\mathcal{X}, T)$  be a dynamical system with a metric $d$. For $N\geq 1$,  we introduce the mean metric $\overline{d}_{N}$ on $\mathcal{X}$ as follows:
$$ \overline{d}_{N}(x, y)=\dfrac{1}{N}\sum\limits_{n=0}^{N-1}\limits d(T^{n}x, T^{n}y).$$
Let $\varphi: \mathcal{X} \rightarrow  \mathbb{R}$ be a continuous function. We define the { $L^{1}$-mean Hausdorff dimension with sub-additive potentials } by
$$ {\rm mdim}_{H,L^{1}}(\mathcal{X}, T, d, \mathcal{F})=\lim\limits_{\epsilon\rightarrow 0} \left( \limsup\limits_{N\rightarrow \infty} \dfrac{{\rm dim}_{H}(\mathcal{X}, \overline{d}_{N},  \varphi_{N}, \epsilon)}{N}\right) .$$
Since $\overline{d}_{N}\leq d_{N}$, we always have
$$ {\rm mdim}_{H,L^{1}}(\mathcal{X}, T, d, \mathcal{F})\leq  {\rm mdim}_{H}(\mathcal{X}, T, d, \mathcal{F}). $$
\begin{lem}\label{le2}
	If $(X, d)$ has the tame growth of covering numbers  and  there exists $K>0$ such that
	$|\varphi_{n+1}(x)-\varphi_{n}(x)|\leq K, ~\forall x\in \mathcal{X}~,n\in \mathbb{N}.$ Then
	$$  {\rm mdim}_{H,L^{1}}(\mathcal{X}, T, d, \mathcal{F})= {\rm mdim}_{H}(\mathcal{X}, T, d, \mathcal{F}).$$
\end{lem}
\begin{proof}
	It is enough to prove ${\rm mdim}_{H}(\mathcal{X}, T, d, \mathcal{F})\leq {\rm mdim}_{H,L^{1}}(\mathcal{X}, T, d, \mathcal{F})$. We use  the notation $[N]:=\left\lbrace 0,1,2, \cdots, N-1\right\rbrace $ and $d_{A}(x, y):=\max_{a\in A}d(T^{a}x, T^{a}y)$ for $A\subset [N]$.
	
	Let $0<\delta<1/2$ and $s>{\rm mdim}_{H,L^{1}}(\mathcal{X}, T, d, \mathcal{F})$ be arbitrary. For each $\tau>0$ we choose an open cover $\mathcal{X}=W_{1}^{\tau}\cup\cdots W_{M(\tau)}^{\tau}$ with ${\rm diam}(W_{i}^{\tau}, d)< \tau$ and $M(\tau)=\#(\mathcal{X}, d, \tau)$. From the tame growth condition, we can find  $0<\epsilon_{0}<1$ such that
	\begin{align}
	M(\tau)^{\tau^{\delta}}<2~~(\forall ~0<\tau<\epsilon_{0}),
	\end{align}
	\begin{align}\label{ieq3}
	2^{2+\delta+(1+2\delta)(s+K+\left\|  \varphi_{1}\right\|_{\infty} )}\cdot \epsilon_{0}^{\delta(1-\delta)}<1.
	\end{align}
	Let $0<\epsilon<\epsilon_{0}$ be a sufficiently small number, and let $N$ be a sufficiently large natural number. Since ${\rm mdim}_{H,L^{1}}(\mathcal{X}, T, d, \mathcal{F})<s$, there exists a covering $\mathcal{X}=\bigcup\limits_{n=1}^{\infty}E_{n}$ with  $\tau_{n}:={\rm diam}(E_{n}, \overline{d}_{N})<\epsilon$ satisfying 	 \begin{align}\label{ieq2}
	\sum\limits_{ i=1}^{\infty} \tau_{n}^{sN-\sup_{E_n} \varphi_{N}}<1, ~~(sN \geq \max\limits_{\mathcal{X}} \varphi_{N}).
	\end{align}
	Set $L_{n}=(1/\tau_{n})^{\delta}$ and pick a point $x_{n} \in E_{n}$ for each  $n$. Then every $x\in E_{n}$ satisfies $\overline{d}_{N}(x, x_{n})<\tau_{n}$ and hence
	$$ \left| \left\lbrace k\in[N]| d(T^{k}x, T^{k}y)\geq L_{n}\tau_{n} \right\rbrace  \right| \leq \dfrac{N}{L_{n}}. $$
	So there exists $A\subset [N]$ (depending on $x\in E_{n}$) such that $\left| A \right| \leq N/L_{n} $ and $d_{[N]\setminus A}(x, x_{n})< L_{n}\tau_{n}.$ Thus
	$$E_{n} \subset \bigcup\limits_{ A\subset [N], \left| A \right|\leq N/L_{n} } B_{L_{n}\tau_{n}}^{\circ}(x_{n}, d_{[N]\setminus A}),$$ where $B_{L_{n}\tau_{n}}(x_{n}, d_{[N]\setminus A})$ is the open ball of radius $L_{n} \tau_{n}$ around $x_{n}$ with respect to  the metric $d_{[N]\setminus A}$.
	
	Let $A=\left\lbrace a_{1}, \cdots, a_{r} \right\rbrace $. We consider a decomposition
	$$B_{L_{n}\tau_{n}}^{\circ}(x_{n}, d_{[N]\ A})=\bigcup\limits_{1\leq i_{1}, \cdots, i_{r}\leq M(\tau_{n})} B_{L_{n}\tau_{n}}^{\circ}(x_{n}, d_{[N]\setminus A})\cap T^{-a_{1}} W_{i_{1}}^{\tau_{n}} \cap \cdots \cap T^{-a_{r}} W_{i_{r}}^{\tau_{n}}. $$
	Then $\mathcal{X}$ is covered by the sets
	\begin{align}\label{eq3}
	E_{n}\cap B_{L_{n}\tau_{n}}^{\circ} (x_{n}, d_{N \setminus A}) \cap T^{-a_{1}} W_{i_{1}}^{\tau_{n}}\cap \cdots \cap T^{-a_{r}} W_{i_{r}}^{\tau_{n}},
	\end{align}
	where $n\geq 1, A=\left\lbrace a_{1}, \cdots, a_{r} \right\rbrace \subset [N] $ with $r\leq N/ L_{n}$ and $1\leq i_{1},\cdots, i_{r}\leq M(\tau_{n})$.	  The sets $(\ref{eq3})$ have diameter less than or equal to $2L_{n}\tau_{n}=2\tau_{n}^{1-\delta}<2\epsilon^{1-\delta}$ with respect to  the metric $d_{N}$.
	Set  $m_{N}=\min\limits_{\mathcal{X}}  \varphi_{N}$. We estimate the quantity
	$$ H_{2\epsilon^{1-\delta}}^{sN+2\delta(sN- m_{N})+\delta N}(\mathcal{X}, d_{N},  \varphi_{N}).$$	
	This is bounded by
	$$ \sum\limits_{n=1}^{\infty} 2^{N}\cdot M(\tau_{n})^{N/ L_{n}}\cdot (2\tau_{n}^{1-\delta})^{sN+2\delta(sN- m_{N})+\delta N-\sup_{E_n} \varphi_{N}}.$$
	The factor $2^{N}$ comes from the choice of $A\subset [N]$. Since $\tau_{n} < \epsilon< \epsilon_{0}$
	\begin{align*}
	(2 \tau_{n}^{1-\delta})^{s N+2\delta(sN- m_{N})+\delta N-\sup_{E_n}  \varphi_{N}}&=(2 \tau_{n}^{1-\delta})^{sN+2\delta(sN-m_{N})-\sup_{E_n}  \varphi_{N}}\cdot(2\tau_{n}^{1-\delta})^{\delta N}\\&\leq(2 \tau_{n}^{1-\delta})^{sN+2\delta(sN-m_{N})-\sup_{E_n}  \varphi_{N}}\cdot (2^{\delta}\epsilon_{0}^{\delta(1-\delta)})^{N}.
	\end{align*}
	The term $(2 \tau_{n}^{1-\delta})^{sN+2\delta(sN-m_{N})-\sup_{E_n}  \varphi_{N}}$ is equal to
	$$\underbrace{2^{sN+2\delta(sN-m_{N})-\sup_{E_{n}} \varphi_{N}}}_{I}\cdot\underbrace{\tau_{n}^{2\delta(sN- m_{N})-\delta\left\lbrace sN+2\delta(sN- m_{N}) -\sup_{E_{n}}\varphi_{N}\right\rbrace}}_{II}\cdot\tau_{n}^{sN-\sup_{E_{n}}\varphi_{N}}.$$
	The factor $(I)$ is bounded by
	$$ 2^{sN+2\delta(sN+KN+\left\| \varphi_{1} \right\|_{\infty}N )+KN+\left\| \varphi_{1}\right\|_{\infty}N }=2^{(1+2\delta)(s+\left\| \varphi_{1} \right\|_{\infty}+K )N}.$$
	The exponent of the factor $(II)$ is bounded from below (note $0<\tau_{n}<1$) by
	$$ 2\delta (sN- m_{N})-\delta\left\lbrace sN+2\delta(sN- m_{N}) -m_{N}\right\rbrace=\delta(1-2\delta)(sN- m_{N})\geq 0.$$
	Here we have used $sN\geq \max_{ \mathcal{X}}\varphi_{N}\geq m_{N}$. Hence the factor $(II)$ is less than or equal to 1. Summing up the above  estimates, we get
	$$ (2 \tau_{n}^{1-\delta})^{s N+2\delta(sN- m_{N})+\delta N-\sup_{E_n}  \varphi_{N}}\leq 2^{(1+2\delta)(s+K+\left\| \varphi_{1} \right\|_{\infty} )N}\cdot(2^{\delta}\epsilon_{0}^{\delta(1-\delta)})^{N}\cdot\tau_{n}^{sN-\sup_{E_{n}} \varphi_{N}}.$$
	Thus
	\begin{align*}
	&H_{2\epsilon^{1-\delta}}^{sN+2\delta(sN- m_{N})+\delta N}(\mathcal{X}, d_{N},  \varphi_{N}) \\&\leq \sum\limits_{n=1}^{\infty}\left\lbrace 2^{1+(1+2\delta)(s+\left\| \varphi_{1} \right\|_{\infty} )}\cdot M(\tau_{n})^{1/ L_{n}}\cdot (2^{\delta}\epsilon_{0}^{\delta(1-\delta)}) \right\rbrace^{N}\cdot\tau_{n}^{sN-\sup_{E_{n}} \varphi_{N}}
	\\&\leq  \sum\limits_{n=1}^{\infty}\left\lbrace 2^{2+\delta +(1+2\delta)(s+\left\| \varphi_{1} \right\|_{\infty} )}\cdot  (\epsilon_{0}^{\delta(1-\delta)}) \right\rbrace^{N}\cdot\tau_{n}^{sN-\sup_{E_{n}} \varphi_{N}}
	\\&\leq \sum\limits_{n=1}^{\infty}\tau_{n}^{sN-\sup_{E_{n}} \varphi_{N}}~~(by~(\ref{ieq3}))\\&<1~~({by}~(\ref{ieq2}))
	\end{align*}
	Therefore
	\begin{align*}
	{\rm dim}_{H}(\mathcal{X},d_{N},  \varphi_{N}, 2\epsilon^{1-\delta})&\leq sN+2\delta(sN- m_{N})+\delta N\\&\leq sN+2\delta(sN+KN+\left\| \varphi_{1} \right\|_{\infty}N )+\delta N.
	\end{align*}
	Divide this by $N$. Let $N\rightarrow \infty$ and $\epsilon\rightarrow 0$:
	$${\rm mdim}_{H}(\mathcal{X}, T, d, \mathcal{F})\leq s+2\delta(s+K+\left\| \varphi_{1} \right\|_{\infty} )+\delta.$$
	Let $\delta\rightarrow 0$ and $s\rightarrow {\rm mdim}_{H, L^{1}}(\mathcal{X}, T, d, \mathcal{F}):$
	$$ {\rm mdim}_{H}(\mathcal{X}, T, d, \mathcal{F})\leq {\rm mdim}_{H, L^{1}}(\mathcal{X}, T, d, \mathcal{F}).$$
\end{proof}
Let $(X, d)$  be a compact metric space. For $\epsilon>0$ and $s\geq 0$ we set $H_{\epsilon}^{s}(\mathcal{X}, d)=H_{\epsilon}^{s}(\mathcal{X},d,0).$ Namely
$$ H_{\epsilon}^{s}(\mathcal{X},d)=\inf\left\lbrace \sum\limits_{i=1}^{\infty}({\rm diam}E_{i})^{s}| \mathcal{X}=\bigcup\limits_{i=1}^{\infty} E_{i}~ \text{with}~ {\rm diam} E_{i}< \epsilon~\forall ~i\geq 1\right\rbrace. $$
We define ${\rm dim}_{H}(\mathcal{X}, d, \epsilon)$ as the supremum of $s\geq 0$ satisfying $H_{\epsilon}^{s}(\mathcal{X}, d)\geq 1$.
\begin{lem}\cite{LT}\label{lem3}
	Let $0<c<1$. There exists $0<\delta_{0}<1$ depending only on $c$ and satisfying the following statement. For any compact metric space $(\mathcal{X}, d)$ and $
	0<\delta<\delta_{0}(c)$ there exists a Borel probability measure $\nu$ on $\mathcal{X}$ such that
	$$ \nu(E)\leq ({\rm diam}E)^{c\cdot dim_{H}(\mathcal{X}, d, \delta)} ~~\text{for all} ~E\subset \mathcal{X}~with ~{\rm diam}E<\dfrac{\delta}{6}. $$
\end{lem}
\begin{lem}\cite{Cao}\label{lem2}
	Suppose $\left\lbrace \nu_{n} \right\rbrace_{n}^{\infty} $ is  a sequence in $\mathcal{M}(\mathcal{X})$, where $\mathcal{M}(\mathcal{X})$ denotes the space of all Borel probability measures on $\mathcal{X}$ with the $weak^{*}$ topology. We form the new sequence $\left\lbrace \mu_{n}\right\rbrace_{n=1}^{\infty} $ by $\mu_{n}=\dfrac{1}{n}\sum_{i=0}^{n-1}\nu_{n}\circ T^{-i}.$ Assume that $\mu_{n_{i}}$ converges to $\mu$ in $\mathcal{X}$ for some subsequence $\left\lbrace n_{i} \right\rbrace $ of natural numbers. Then $\mu\in \mathcal{M}(\mathcal{X}, T)$, and moreover
	$$ \limsup\limits_{i\rightarrow \infty}\dfrac{1}{n_{i}}\int \log f_{n_{i}} d \nu_{n_{i}}\leq \mathcal{F}_{*}(\mu).$$
\end{lem}
\begin{lem}\label{lem}
	Let $A$ be a finite set. Suppose that probability measures $\mu_{n}$ on $A$ converge to some $\mu$ in the $weak^{*}$ topology. Then there exist probability measures $\pi_{n} (n\geq 1)$ on $A\times A$ such that
	\begin{itemize}
		\item $\pi_{n}$ is a coupling between $\mu_{n}$ and $\mu$. Namely the first and second marginals of $\pi_{n}$ are given by $\mu_{n}$ and $\mu$ respectively.
		\item $\pi_{n}$ converge to $(id\times id)_{*}\mu$ in the $weak^{*}$ topology. Namely
		\begin{equation*}
		\pi_{n}(a, b)\rightarrow
		\begin{cases}0, ~~~~~\text{if}~ (a\neq b), \\[5pt] \mu(a), ~~~\text{if}~ (a=b).\\
		\end{cases}
		\end{equation*}
	\end{itemize}
\end{lem}
\begin{theorem}\label{thmc}
	Assume that $\overline{{\rm mdim_{M}}}(\mathcal{X}, T, d)<\infty$ for all $d\in \mathcal{D}(X)$ and there exists $K>0$ such that
	$|\varphi_{n+1}(x)-\varphi_{n}(x)|\leq K, ~\forall x\in \mathcal{X}~,n\in \mathbb{N}.$	Under a mild condition on $d$ (called tame growth  of covering numbers)
	$$ {\rm mdim}_{H}(\mathcal{X}, T, d, \mathcal{F}) \leq \sup\limits_{\mu \in  \mathcal{M}(\mathcal{X}, T)}(\underline{{\rm rdim}}(\mathcal{X}, T, d, \mu)+ \mathcal{F}_{*}(\mu)).$$	
\end{theorem}
The Theorem \ref{thmc} follows from Lemma \ref{le2} and Theorem \ref{ss}.
\begin{theorem}\label{ss}
	Assume that $\overline{{\rm mdim_{M}}}(\mathcal{X}, T, d)<\infty$ for all $d\in \mathcal{D}(X)$ and there exists $K>0$ such that
	$|\varphi_{n+1}(x)-\varphi_{n}(x)|\leq K, ~\forall x\in \mathcal{X}~,n\in \mathbb{N}$. For any dynamical system $(\mathcal{X}, d)$ with metric $d$, then
	$${\rm mdim}_{H, L^{1}}(\mathcal{X}, T, d, \mathcal{F}) \leq \sup\limits_{\mu \in  \mathcal{M}(\mathcal{X}, T)}(\underline{{\rm rdim}}(\mathcal{X}, T, d, \mu)+ \mathcal{F}_{*}(\mu)).$$
\end{theorem}
\begin{proof}[Proof of Theorem \ref{ss}]
	We extend the definition of $\overline{d}_{n}$. For $x=(x_{0}, x_{1}, \cdots, x_{n-1})$ and $y=(y_{0}, y_{1}, \cdots, y_{n-1})$ in $\mathcal{X}^{n}$, we set $$\overline{d}_{n}(x, y)=\dfrac{1}{n}\sum\limits_{i=0}^{n-1}d(x_{i}, y_{i}).$$
	
	Let $0<c<1$ and $s<{\rm mdim}_{H, L^{1}}(\mathcal{X}, T, d, \mathcal{F})$ be arbitrary. Then there exists  an invariant probability measure $\mu$ on $\mathcal{X}$ such that
	\begin{align}\label{eq4}
	\underline{{\rm rdim}}(\mathcal{X}, T, d, \mu)+\mathcal{F}_{*}(\mu)\geq cs-(1-c)\left\|  \varphi_{1}\right\|_{\infty}.
	\end{align}
	Take $\eta>0$ satisfying ${\rm mdim}_{H, L^{1}}(\mathcal{X}, T, d, \mathcal{F})-2\eta >s.$ Let $\delta_{0}=\delta_{0}(c)\in (0,1)$ be a constant given by Lemma \ref{lem3}. There exist $0<\delta <\delta_{0}$ and a sequence $n_{1}<n_{2}<n_{3}<\cdots\rightarrow \infty$ satisfying
	$$ {\rm dim}_{H}(\mathcal{X}, \overline{d}_{n_{k}},  \varphi_{n_{k}}, \delta) >(s+2\eta)n_{k}.$$
\begin{claim}
	There exists $t\in \left[-\left\|  \varphi_{1} \right\|_{\infty}-K,  \left\| \varphi_{1} \right\|_{\infty}+K \right] $ such that for infinitely many $n_{k}$
	$$ {\rm dim}_{H}\left( (\dfrac{ \varphi_{n_{k}}}{n_{k}})^{-1}[t, t+\eta], \overline{d}_{n_{k}}, \delta\right) \geq (s-t)n_{k}.$$
\end{claim}
\begin{proof}
	Since ${\rm dim}_{H}(\mathcal{X}, \overline{d}_{n_{k}}, \varphi_{n_{k}}, \delta) >(s+2\eta)n_{k}$, we have $$H_{\delta}^{(s+2\eta)n_{k}}(\mathcal{X}, \overline{d}_{n_{k}}, \varphi_{n_{k}})\geq 1.$$ Set $m=[\dfrac{2\left\|  \varphi_{1}\right\|_{\infty}+2K }{\eta}]$ and consider a decomposition of $\mathcal{X}$, namely,
	$$\mathcal{X}=\bigcup\limits_{l=0}^{m-1}(\dfrac{ \varphi_{n_{k}}}{n_{k}})^{-1}[l\eta, (l+1)\eta].$$
	Then there exists $t\in \left\lbrace -\left\|  \varphi_{1} \right\|_{\infty}-K+l\eta| l=0, 1, \cdots, m-1 \right\rbrace $ such that for infinitely many $n_{k}$
	$$ H_{\delta}^{(s+2\eta)n_{k}}\left( (\dfrac{ \varphi_{n_{k}}}{n_{k}} )^{-1}[t, t+\eta], \overline{d}_{n_{k}},  \varphi_{n_{k}}\right)\geq \dfrac{1}{m}.$$
	Since $(s+2\eta)n_{k}- \varphi_{n_{k}}\geq (s+2\eta)n_{k}-(t+\eta)n_{k}=(s-t)n_{k}+\eta n_{k}$ on  the set $( \varphi_{n_{k}}/n_{k} )^{-1}[t, t+\eta],$
	\begin{align*}
	H_{\delta}^{(s+2\eta)n_{k}}((\dfrac{ \varphi_{n_{k}}}{n_{k}})^{-1}[t, t+\eta], \overline{d}_{n_{k}},  \varphi_{n_{k}})&\leq H_{\delta}^{(s-t)n_{k}+\eta n_{k}}((\dfrac{ \varphi_{n_{k}}}{n_{k}})^{-1}[t, t+\eta], \overline{d}_{n_{k}})\\&\leq \delta^{\eta n_{k}}\cdot H_{\delta}^{(s-t)n_{k}}((\dfrac{ \varphi_{n_{k}}}{n_{k}})^{-1}[t, t+\eta], \overline{d}_{n_{k}}).
	\end{align*}
	Hence for infinitely many $n_{k}$,
	$$ H_{\delta}^{(s-t)n_{k}}\left( (\dfrac{ \varphi_{n_{k}}}{n_{k}})^{-1}[t, t+\eta], \overline{d}_{n_{k}}\right) \geq \dfrac{\delta^{-\eta n_{k}}}{m}.$$
	The right-hand side is large than one for sufficiently large $n_{k}$. Then for such $n_{k}$
	$${\rm dim}_{H}\left( (\dfrac{ \varphi_{n_{k}}}{n_{k}})^{-1}[t, t+\eta], \overline{d}_{n_{k}}, \delta\right) \geq  (s-t)n_{k}.$$
\end{proof}
By  choosing a subsequence of $n_{k}$ (also denoted by $\left\lbrace n_{k} \right\rbrace $), we assume that the condition
$$ {\rm dim}_{H}\left( (\dfrac{ \varphi_{n_{k}}}{n_{k}})^{-1}[t, t+\eta]\right) , \overline{d}_{n_{k}}, \delta)\geq (s-t)n_{k}$$
holds for all $n_{k}$. Noting that $0<\delta<\delta_{0}(c)$, we apply Lemma \ref{lem3} to  the subspace $(\dfrac{ \varphi_{n_{k}}}{n_{k}})^{-1}[t, t+\eta]\subset \mathcal{X}$. Then we can find a Borel probability measure $\nu_{k}$ supported on $(\dfrac{ \varphi_{n_{k}}}{n_{k}})^{-1}[t, t+\eta]$ such that
\begin{align}\label{eq7}
\nu_{k}(E)\leq ({\rm diam}(E, \overline{d}_{n_{k}}))^{c(s-t)n_{k}}~~\text{for all}~E\subset \mathcal{X}~\text{with} ~~{\rm diam}(E, \overline{d}_{n_{k}})<\dfrac{\delta}{6}.
\end{align}
Notice that $\nu_{k}$ is not necessarily invariant under $T$. Set
$$ \mu_{k}=\dfrac{1}{n_{k}}\sum\limits_{n=0}^{n_{k}-1}T_{*}^{n}\nu_{k}.$$
By choosing a subsequence (also denoted by $\left\lbrace n_{k} \right\rbrace $ again) we can assume that $\mu_{k}$ converges to some $\mu \in \mathcal{M}(\mathcal{X}, T)$ in the $weak^{*}$ topology. By Lemma \ref{lem2}
$$\limsup\limits_{k\rightarrow \infty}\dfrac{1}{n_{k}} \int_{\mathcal{X}} \varphi_{n_{k}}d\nu_{k}\leq \mathcal{F}_{*}(\mu)=\lim\limits_{k\rightarrow \infty}\dfrac{1}{n_{k}} \int_{\mathcal{X}} \varphi_{n_{k}} d \mu.$$
On the other hand
$$\int_{\mathcal{X}} \dfrac{ \varphi_{n_{k}}}{n_{k}} d \nu_{k} \geq t $$
since $\nu_{k}$ is supported on the set $(\dfrac{ \varphi_{n_{k}}}{n_{k}})^{-1}[t, t+\eta]$. Hence
$ \mathcal{F}_{*}(\mu)\geq t.$ Moreover, since $0\leq\overline{\rm{rdim}
}(\mathcal{X}, T, d, \mu)<\infty$. Then  we need to  prove
\begin{align}\label{eq5}
\underline{\rm rdim}(\mathcal{X}, T, d, \mu )\geq c(s-t).
\end{align}
If the above inequality holds, we  will get (\ref{eq4}) (recall $\left| t \right|\leq \left\| \varphi_{1} \right\|_{\infty}+K  $):
$$\underline{{\rm rdim}}(\mathcal{X}, T, d, \mu)+\mathcal{F}_{*}(\mu)\geq c(s-t)+t=cs+(1-c)t \geq cs-(1-c)(\left\| \varphi_{1} \right\|_{\infty}+K).$$
So the rest of the problem is to prove $(\ref{eq5}) .$ This part  of the proof is the same as \cite{LT}. The method is  a " rate distortion theory version" of Misiurewicz's technique \cite{Mis76} (a famous proof of the standard variational principle) first developed in \cite{LT18}. The paper \cite{LT} explained more background ideas  behind the proof, which we do not repeat here.

Let $\epsilon$ be an arbitrary positive number with $2 \epsilon \log (1/\epsilon)\leq \delta/ 10. $ We will show a lower bound on the rate distortion function of the form
$$ R(d, \mu, \epsilon)\geq c(s-t)\log (1/\epsilon)+ \text{small error terms.}$$
Let $X$ and $Y=(Y_{0}, Y_{1}, \cdots, Y_{m-1})$ be random variables defined on a probability space $(\Omega, \mathbb{P})$ such that $X, Y_{0}, \cdots, Y_{m-1}$ take values in $\mathcal{X}$ and satisfy
$$ {\rm Law}(X)=\mu, ~~\mathbb{E}\left( \dfrac{1}{m}\sum\limits_{j=0}^{m-1}d (T^{j}X, Y_{j})\right) <\epsilon.$$
We would like  to establish  a lower bound on the mutual information $I(X; Y)$. For this purpose, we can assume that $Y$ takes only finitely  many values. Let $\mathcal{Y} \subset \mathcal{X}^{m}$ be the (finite) set of possible values of $Y$.

We choose $\tau >0$ satisfying
\begin{align}\label{eqa}
\tau \leq  \min(\dfrac{\epsilon}{3}, \dfrac{\delta}{20}),~~\dfrac{\tau}{2}+\mathbb{E}\left( \dfrac{1}{m}\sum_{j=0}^{m-1}d(T^{j}X, Y_{j})\right) <\epsilon.
\end{align}
We take a measurable partition $\mathcal{P}=\left\lbrace P_{1}, \cdots, P_{L} \right\rbrace $ of $\mathcal{X}$ such that for all $1\leq l \leq L$
$$ {\rm diam}(P_{l}, d)< \dfrac{\tau}{2},~~\mu(\partial P_{l})=0.$$
We choose a point $p_{k} \in P_{k}$ for each
$1\leq k \leq K$. Set $A= \left\lbrace p_{1}, \cdots, p_{K} \right\rbrace $. We define  a map $\mathcal{P}: \mathcal{X} \rightarrow A$ by $\mathcal{P}(x)=p_{k}$ for $x \in P_{k}$. It follows that
\begin{align}\label{eq6}
d(x, \mathcal{P}(x))< \epsilon.
\end{align}
For $n \geq 1$, we set
$\mathcal{P}^{n}(x)= (\mathcal{P}(x), \mathcal{P}(T(x)), \cdots, \mathcal{P}(T^{n-1}x)).$
\begin{claim}\label{cla}
	The pushforward measure $\mathcal{P}_{*}^{n_{k}}\nu_{k}$ satisfies
	$$ \mathcal{P}_{*}^{n_{k}}\nu_{k}(E)\leq (\tau + {\rm diam}(E, \overline{d}_{n_{k}}))^{c(s-t)n_{k}}~~\text{for all}~E\subset A^{n_{k}}~\text{with}~{\rm diam}(E, \overline{d}_{n_{k}})<\dfrac{\delta}{10}.$$
\end{claim}
\begin{proof}
	From ${\rm diam}(P_{l}, d)<\tau/2$ and $\tau \leq \delta/20,$ if ${\rm diam}(E, \overline{d}_{n_{k}})<\delta /10$ then
	$${\rm diam}((\mathcal{P}_{n_{k}})^{-1}E, \overline{d}_{n_{k}})<\tau+{\rm diam}(E, \overline{d}_{n_{k}})<\dfrac{\delta}{6}.$$
	By $(\ref{eq7})$, the measure $\mathcal{P}_{*}^{n_{k}}(E)=\nu_{k}((\mathcal{P}^{n_{k}})^{-1} E)$ is bounded by
	$$ ({\rm diam}((\mathcal{P}^{n_{k}})^{-1} E), \overline{d}_{n_{k}})^{c(s-t)n_{k}}< (\tau + {\rm diam}(E, \overline{d}_{n_{k}}))^{c(s-t)n_{k}}.$$
\end{proof}
From $\mu_{k}\rightarrow \mu$ and $\mu(\partial P_{l})=0,$ we have $\mathcal{P}_{*}^{m}\mu_{k}\rightarrow \mathcal{P}_{*}^{m}\mu$. By Lemma {\ref{lem}}, there exists a coupling $\pi_{k}$ between $\mathcal{P}_{*}^{m}\mu_{k}$ and $\mathcal{P}_{*}^{m}\mu$ such that $\pi_{k}\rightarrow (id \times id)_{*}\mathcal{P}_{*}^{m}\mu.$
Let $X(k)$ be a random variable couple to $\mathcal{P}^{m}(X)$ such that  it takes values in $A^{m}$ and  Law $(X(k), \mathcal{P}^{m}(X))=\pi_{k}.$ In particular, ${\rm Law}X(k)=\mathcal{P}_{*}^{m}\mu_{k}.$ From $\pi_{k} \rightarrow (id \times id)_{*}\mathcal{P}_{*}^{m}\mu,$
$$\mathbb{E}\overline{d}_{m}(X(k), \mathcal{P}^{m}(X))\rightarrow 0.$$
The random variables $X(k)$ and $Y$ are coupled by the probability mass function
$$ \sum \limits_{x'\in A^{m}}\pi_{k}(x, x')\mathbb{P}(Y=y|\mathcal{P}^{m}(X)=x')~~(x\in A^{m}, y \in \mathcal{Y}),$$
which converges to $\mathbb{P}(\mathcal{P}^{m}(X)=x, Y=y).$ Then by Lemma \ref{le1},
\begin{align}\label{lim}
I(X(k); Y)\rightarrow I(\mathcal{P}^{m}(X); Y).
\end{align}
By the triangle inequality
\begin{align*}
\overline{d}_{m}(X(k), Y)\leq& \overline{d}_{m}(X(k), \mathcal{P}^{m}(X))+ \overline{d}_{m}(\mathcal{P}^{m}(X), (X, TX, \cdots, T^{m-1}X))\\&+\overline{d}_{m}((X, TX, \cdots, T^{m-1}X), Y)
\end{align*}
We have $\mathbb{E}\overline{d}_{m}(X(k), \mathcal{P}^{m}(X))\rightarrow 0$, ${\rm diam}(P_{l}, d)<\tau/ 2$ for all $1\leq l \leq L$ and $\tau /2+ \mathbb{E} \overline{d}_{m}((X, TX, \cdots, T^{m-1}X), Y)<\epsilon$ in (\ref {eqa}). Then
\begin{align}\label{ieq}
\mathbb{E} \overline{d}_{m}(X(k), Y)<\epsilon ~~\text{for sufficiently large}~ k
\end{align}

Let $n_{k}=qm+r$ with $ m\leq r \leq 2m-1.$ Fix a point $a\in \mathcal{X}$. We denote by $\delta_{a}(\cdot)$ the delta probability  measure at $a$ on $\mathcal{X}$. For $x\in (x_{0}, \cdots, x_{n-1}) \in \mathcal{X}^{n}$, we let $x_{k}^{l}$ denote the $(l-k+1)$-tuple $x_{k}^{l}=(x_{k}, \cdots, x_{l})$  for $0\leq k \leq l <n$.  We consider a conditional probability mass function
$$ \rho_{k}(y|x)=\mathbb{P}(Y=y| X(k)=x)$$ for $x, y \in \mathcal{X}^{m}$ with $\mathbb{P}(X(k)=x)=\mathcal{P}_{*}^{m}\mu_{k}(x)>0.$   We define probability mass functions $\sigma_{k,0}(\cdot| x), \cdots, \sigma_{k,m-1}(\cdot| x)$ on $\mathcal{X}^{n}$ by
\begin{align}\label{3.9}
\sigma_{k,j}=\prod_{j=0}^{q-1} \rho_{k}(y_{j+im}^{j+im+m-1}| x_{j+im}^{j+im+m-1})\times \prod_{n \in [0, j)\cup [mq+j, n_{k}]}\delta_{a}(y_{n_{k}}).
\end{align}
We set
\begin{align}\label{eq}
\sigma_{k}(y|x)=\dfrac{\sigma_{k,0}(y|x)+\sigma_{k,1}(y|x)+ \cdots + \sigma_{k, m-1}(y|x)}{m}.
\end{align}
Let $X'(k)$ be a random variable taking values in $\mathcal{X}$ with Law$X'(k)=\nu_{k}$. Set $Z(k)=\mathcal{P}^{n_{k}}(X'(k)).$ We define a random variable $W(k)$ taking values in $\mathcal{X}^{n_{k}}$ and coupled to $Z(k)$ by the condition
$$\mathbb{P}(W(k)=y| Z(k)=x)=\sigma_{k}(y|x).$$ For $0\leq j<m$ we also define $W(k, j)$ by
$$ \mathbb{P}(W(k, j)=y| Z(k)=x)=\sigma_{k,j}(y|x).$$
\begin{claim}\label{2}
	$\dfrac{1}{m} I(X_{k}; Y) \geq \dfrac{1}{n} I(Z(k); Y).$
\end{claim}

\begin{proof}
	The  mutual information is a convex function of conditional probability measure (Lemma \ref{lemc}). Hence
	$$ I(Z(k); W(k))\leq \dfrac{1}{m}\sum\limits_{j=0}^{m-1}I( Z(k); W(k, j)).$$
	By the subadditivity under conditional independence (Lemma \ref{lems}),
	$$ I(Z(k); W(k, j))\leq \sum\limits_{i=0}^{q-1}I(Z(k); W(k, j)_{j+im}^{j+im+m-1}).$$
	The  term $I(Z(k); W(k, j)_{j+im}^{j+im+m-1})$ is equal to
	$$ I(\mathcal{P}^{m}(T^{j+im}X'(k); W(k,j)_{j+im}^{j+im+m-1}=I(\mathcal{P}_{*}^{m}T^{j+im}\nu_{k}, \rho_{k}).$$
	Therefore
	\begin{align*}
	\dfrac{m}{n_{k}}I(Z(k), W(k))&\leq \dfrac{1}{n_{k}}\sum\limits_{\substack{0\leq j<m\\ 0\leq i< q}}I(\mathcal{P}_{*}^{m}T^{j+im}\nu_{k}, \rho_{k})\\&\leq \dfrac{1}{n_{k}}\sum\limits_{n=0}^{n_{k}-1}I(\mathcal{P}_{*}^{m}T^{n}\nu_{k}, \rho_{k})\\&\leq I(\dfrac{1}{n_{k}}\sum\limits_{n=0}^{n_{k}-1}\mathcal{P}_{*}^{m}T^{n}\nu_{k}, \rho_{k}) ~\text{by the concavity in Lemma \ref{lemc}}\\&= I(\mathcal{P}_{*}^{m}\mu_{k}, \rho_{k})~\text{by}~\mu_{k}=\dfrac{1}{n_{k}}\sum\limits_{n=0}^{n_{k}-1}T_{*}^{n}\nu_{k}\\&=I(X(k); Y).
	\end{align*}
\end{proof}

\begin{claim}\label{cla1}
	For sufficiently large $k$
	$$\mathbb{E}(\overline{d}_{n_{k}}(Z(k), W(k)))<\epsilon. $$
\end{claim}
\begin{proof}
	By (\ref{eq}), we have	$$\mathbb{E}(\overline{d}_{n}(Z(k), W(k))=\dfrac{1}{m}\sum\limits_{ j=0}^{m-1}\mathbb{E}(\overline{d}_{n_{k}}(Z(k), W(k,j)). $$
	From,  $Z(k)=\mathcal{P}^{n_{k}}(X'(k), W(k, j)_{j+im}^{j+im+m-1})$, the distance $\overline{d}_{n_{k}}(Z(k), W(k, j))$ is bounded by
	$$ \dfrac{r\cdot {\rm diam}(\mathcal{X}, d)}{n_{k}}+\dfrac{m}{n_{k}}\sum\limits_{i=0}^{q-1}\overline{d}_{m}(\mathcal{P}^{m}(T^{j+im} X'(k), W(k, j)_{j+im}^{j+im+m-1})).$$
	$\mathbb{E}\overline{d}_{m}(\mathcal{P}^{m}(T^{j+im} X'(k), W(k, j)_{j+im}^{j+im+m-1})$ is equal to
	$$\sum\limits_{x, y \in \mathcal{X}^{m}}\overline{d}_{m}(x, y)\rho_{k}(y|x)\mathcal{P}^{m}T^{j+im} \nu_{k}(x).$$
	Therefore
	
	\begin{align*}
	\mathbb{E}(\overline{d}_{n}(Z(k), W(k)) &
	\leq \dfrac{r\cdot {\rm diam}(\mathcal{X}, d)}{n_{k}}+\sum\limits_{x, y \in \mathcal{X}^{m}}\overline{d}_{m}(x, y)\rho_{k}(y|x)\left( \dfrac{1}{n_{k}}\sum\limits_{\substack{0\leq j<m\\ 0\leq i< q}}\mathcal{P}_{*}^{m} T_{*}^{j+im}\nu_{k}(x) \right) \\& =\dfrac{r\cdot {\rm diam}(\mathcal{X}, d)}{n_{k}}+\sum\limits_{x, y \in \mathcal{X}^{m}}\overline{d}_{m}(x, y)\rho_{k}(y|x)\left( \dfrac{1}{n_{k}}\sum\limits_{n=0}^{n_{k}-1} \mathcal{P}_{*}^{m} T_{*}^{n}\nu_{k}(x)\right)  \\& = \dfrac{r\cdot {\rm diam}(\mathcal{X}, d)}{n_{k}}+\sum\limits_{x, y \in \mathcal{X}^{m}}\overline{d}_{m}(x, y)\rho_{k}(y|x)  \mathcal{P}_{*}^{m}\mu_{k}(x) \\&=\dfrac{r\cdot {\rm diam}(\mathcal{X}, d)}{n_{k}}+\mathbb{E}\overline{d}_{m}(X(k), Y).
	\end{align*}
	From $r\geq 2m$ and (\ref{ieq}), this is less than $\epsilon$ for large $k$.
\end{proof}
Recall $2\epsilon \log (1/\epsilon)\leq \delta/10$ and $\tau\leq \min(\epsilon/3, \delta/ 20).$ The measure Law $Z(k)=\mathcal{P}_{*}^{n_{k}}\nu_{k}$ satisfies the "scaling law" given by Claim \ref{cla}. Then we apply Lemma \ref{ml2} to $(Z(k), W(k))$ with Claim \ref{cla1}, which provides
\begin{align}
I(Z(k); W(k))\geq c(s-t)n_{k}\log (1/\epsilon)-T(c(s-t)n_{k}+1) ~~\text{for large}~k.
\end{align}
Here $T$ is a universal positive constant. From Claim \ref{2},
$$\dfrac{1}{m}I(X(k); Y)\geq c(s-t) \log (1/\epsilon)-T(c(s-t)+\dfrac{1}{n_{k}}).$$
We know $I(X(k); Y)\rightarrow I(\mathcal{P}^{m}(X); Y)$ as $k\rightarrow \infty$ in (\ref{lim}). Hence
$$ \dfrac{1}{m}I(\mathcal{P}^{m}(X); Y)\geq c(s-t) \log (1/\epsilon)-cT(s-t).$$
By the date-processing inequality (Lemma \ref{dp} )
$$\dfrac{1}{m}I(X;Y)\geq \dfrac{1}{m}I(\mathcal{P}^{m}(X); Y)\leq c(s-t) \log (1/\epsilon)-cT(s-t).$$
This proves that for any $\epsilon>0$ with $2\epsilon\log (1/\epsilon)\leq \delta/10$
$$ R(d, \mu, \epsilon)\geq c(s-t) \log (1/\epsilon)-cT(s-t).$$ Thus we get $(\ref {eq5})$:
$$\underline{\rm rdim}(\mathcal{X}, T, d, \mu)=\liminf\limits_{ \epsilon \rightarrow 0}\dfrac{R(d, \mu, \epsilon)}{\log (1/\epsilon)}\geq c(s-t).$$ This establishes the proof of the theorem.
\end{proof}
\section{Proof of Theorem \ref{aa}}\label{d}
In this section, we give some results on combinatorial topology and dynamical tiling construction. We prove the following conclusion.
\begin{theorem}\label{aa}
	If $(\mathcal{X}, T)$  has the marker property and there exists $K>0$ such that $|\varphi_{n+1}-\varphi_{n}|<K$ for every $n$, then there exists a metric ${ d} \in \mathcal{D}(\mathcal{X})$ metric satisfying
	$$ {\rm mdim}(\mathcal{X}, T, \mathcal{F})=\overline{\rm mdim}_{M}(\mathcal{X}, T, d, \mathcal{F}).$$
\end{theorem}
	\subsection{Preparations on combinatorial topology}
		In this subsection we prepare some definitions and results about simplicial complex. Recall that we have assumed that simplicial complexes are always finite (having only finitely many vertices).
		
		Let $P$ be a simplicial complex. We denote by ${\rm Ver}(P)$ the set of vertices of $P.$ For a vertex $v$ of $P$ we define the $\bf{open\ star}$ $O_{P}(v)$ as the union of open simplexes of $P$ one of whose vertex is $v$.  Here $\{v\}$ itself is an open simplex. So $O_{P}(v)$ is an open neighborhood of $v$, and $\{O_{P}(v)\}_{v\in {\rm Ver}(P)}$ forms an open cover of $P.$ For a simplex $\Delta\subset P$ we set $O_{P}(\Delta)=\bigcup_{v\in {\rm Ver}(\Delta)O_{P}(v)}.$
\begin{defn}
	Let $P$ and $Q$ be simplicial complexes. A map $f: P\rightarrow Q$ is said to be ${simplicial}$ if for every simplex $\Delta\subset P$ the image $f(\Delta)$ is a simplex in $Q$ and
	\begin{equation*}
	f(\sum_{v\in {\rm Ver}(\Delta)}\lambda_{v}v)=\sum_{v\in {\rm Ver}(\Delta)}\lambda_{v}f(v),
	\end{equation*}
	where $0\leq \lambda_{v}\leq 1 $ and $\sum_{v\in {\rm Ver}(\Delta)}\lambda_{v}= 1.$	
\end{defn}
\begin{defn}
	Let $V$	be a real vector. A map $f:P\rightarrow V$ is said to be ${linear}$ if for every simplex $\Delta\subset P$
	\begin{equation*}
	f(\sum_{v\in {\rm Ver}(\Delta)}\lambda_{v}v)=\sum_{v\in {\rm Ver}(\Delta)}\lambda_{v}f(v),
	\end{equation*}
	where $0\leq \lambda_{v}\leq 1 $ and $\sum_{v\in {\rm Ver}(\Delta)}\lambda_{v}= 1.$	
\end{defn}
We denote the space of linear maps $f:\ P\rightarrow V$ by Hom$(P,V).$ When $V$ is a Banach space, the space Hom$(P,V)$ is topologized as a product space $V^{{\rm Ver}(P)}.$
\begin{lem}\label{key}\cite{LT}
	Let $(V, ||\cdot||)$ ba a Banach space and $P$ a simplicial complex.\\
	(1) If $f: P\rightarrow V$ is a linear map with $diam f(P)\leq 2$ then for any $0< \epsilon \leq 1$
	\begin{equation*}
	\#(f(P),||\cdot||,\epsilon)\leq C(P)\cdot(1/\epsilon)^{dim P}.
	\end{equation*}
	Here the left-hand side is the minimum cardinality of open covers $\mathcal{U}$ of $f(P)$ satisfying $diam\ U< \epsilon$ for all $U\in \mathcal{U}$. $C(P)$ is a positive constant depending only on $dim P$ and the number of somplexes of $P.$\\
	(2) Suppose $V$ is infinite dimensional. Then the set
	\begin{equation}
	\{f\in Hom(P,V)|  f\ is\ \text{injective}\}\label{xiaoming}
	\end{equation}
	is dense in Hom(P,V).\\
	(3) Let $(\mathcal{X}, {d})$ be a compact metric space and $\epsilon,\  \delta> 0.$ Let $\pi:\ \mathcal{X}\rightarrow P$ be a continuous map satisfying $diam \pi^{-1}(O_{P}(v))< \epsilon$ for all $v\in {\rm Ver}(P).$  Let $\pi:\ \mathcal{X}\rightarrow V$ be a continuous map such that
	\[
	{{d}} (x,y)<\epsilon   \Longrightarrow ||f(x)-f(y)||<\delta.
	\]
	Then there exists a linear map $g:\ P\rightarrow V$ satisfying
	\begin{equation*}
	||f(x)-g(\pi(x))||< \delta
	\end{equation*}
	for all $x\in \mathcal{X}.$ Moreover if $f(\mathcal{X})$ is contained in the open unit ball $B_{1}^{o}(V)$ then we can assume $g(P)\subset B_{1}^{o}(V).$
\end{lem}
 \begin{defn}
	Let $f:\ \mathcal{X}\rightarrow P$ be a continuous map from a topological space $\mathcal{X}$ to a simplicial complex $P.$  It is said to be essential if there is no proper subcomplex of $P$ containing $f(\mathcal{X})$. This is equivalent to the condition that for any simplex $\Delta \subset P$
	\begin{equation*}
	\bigcap_{v\in {\rm Ver}(\Delta)}f^{-1}(O_{P}(v))\neq \emptyset.
	\end{equation*}
\end{defn}
	\begin{lem}\label{lem4}\cite{LT}
	Let $f:\ \mathcal{X}\rightarrow P$ be a continuous map from a topological space $\mathcal{X}$ to a simplicial complex $P.$ There exists a subcomplex $P'\subset P$ such that $f(\mathcal{X})\subset P'$ and $f:\ \mathcal{X}\rightarrow P'$ is essential.
\end{lem}
For two open covers $\mathcal{U}$ and $\mathcal{V}$ of $\mathcal{X}$, we say that $\mathcal{V}$ is refinement of $\mathcal{U}$ (denoted by $\mathcal{U}\prec \mathcal{V}$) if for every $V\in \mathcal{V}$ there exists $U\in \mathcal{U}$ containing $V$.
\begin{lem}\label{lem5.3}\cite{LT}
	Let $\mathcal{X}$ be a topological space, $P$ and $Q$ simplicial complexes. Let $\pi:\ \mathcal{X}\rightarrow P$
	and $q_{i}:\ \mathcal{X}\rightarrow Q$ $(1\leq i\leq N)$ be continuous maps. We suppose that $\pi$ is essential and satisfies for all $1\leq i\leq N$
	\begin{equation*}
	\{q_{i}^{-1}(O_{Q}(w))\}_{w\in {\rm Ver}(Q)}\prec \{\pi^{-1}(O_{P}(v))\}_{v\in {\rm Ver}(P)}\ (\text{as open covers of}\ \mathcal{X}).
	\end{equation*}
	Then there exist simplicial maps $h_{i}:\ P\rightarrow Q$ $(1\leq i\leq N)$ satisfying the following three conditions.\\
	(1) For all $1\leq i\leq N$ and $x\in \mathcal{X}$ the two points $q_{i}(x)$ and $h_{i}(\pi(x))$ belong to the same complex of $Q$.\\
	(2) Let $1\leq i\leq N$ and let $Q' \subset Q$ be a subcomplex. If a simplex $\Delta \subset P$ satisfies $\pi^{-1}(O_{P}(\Delta))\subset q_{i}^{-1}(Q')$ then $h_{i}(\Delta)\subset Q'$.\\
	(3) Let $\Delta \subset P$ be a simplex. If $q_{i}=q_{j}$ on $\pi^{-1}(O_{P}(\Delta))$ then $h_{i}=h_{j}$ on $\Delta$.
\end{lem}
\subsection{Dynamical tiling construction}\label{5.2}
The purpose of this subsection is to define a "dynamical decomposition" of the real line, which was first introduced in \cite{GLT16}. This will be the basis of the construction in the proof of Theorem 1.8.

Let $(\mathcal{X}, T)$ be a dynamical system and $\psi:\ \mathcal{X}\rightarrow [0,1]$ a continuous function. Take $x\in \mathcal{X}$. We consider
\begin{equation}\label{xiaohua}
\{(a, \frac{1}{\psi(T^{a}x)})| a\in \mathbb{Z} \ \text{with} \ \psi(T^{a}x)> 0\}.
\end{equation}
This is a discrete subset of the plane. We assume that (\ref{xiaohua}) is nonempty for every $x\in \mathcal{X}$.
Namely for every $x\in \mathcal{X}$ there exists a $a\in \mathbb{Z}$ with $\psi(T^{a}x)> 0$. Let $\mathbb{R}^{2}=\bigcup_{a\in \mathbb{Z}}V_{\psi}(x, a)$ be the associated $\bf{Voronoi\ diagram},$ where $V_{\psi}(x,a)$ is the (convex) set of $u\in \mathbb{R}^{2}$ satisfying
\begin{equation*}
|u-(a,\frac{1}{\psi(T^{a}x)})|\leq |u-(b, \frac{1}{\psi(T^{b}x)})|
\end{equation*}
for any $b\in \mathbb{Z}$ with $\psi(T^{b}x)> 0.$ (If $\psi(T^{a}x)=0$ then $V_{\psi}(x,a)$ is empty.) We set
\begin{equation*}
I_{\psi}(x,a)= V_{\psi}(x,a)\cap (\mathbb{R}\times \{0\}).
\end{equation*}
See Figure in \cite{LT}.
We naturally identity $\mathbb{R}\times \{0\}$ with $\mathbb{R}.$ This provides a decomposition of $\mathbb{R}:$
\begin{equation*}
\mathbb{R}= \bigcup_{a\in \mathbb{Z}}I_{\psi}(x,a).
\end{equation*}
We set
\begin{equation*}
\partial_{\psi}(x)=\bigcup_{a\in \mathbb{Z}}\partial I_{\psi}(x,a)\subset \mathbb{R},
\end{equation*}
where $\partial I_{\psi}(x,a)$ is the boundary of $I_{\psi}(x,a)$ (e.g. $\partial[0,1]=\{0,1\}$). This construction is equivariant:
\begin{equation*}
I_{\psi}(T^{n}x, a)= -n+I_{\psi}(x, a+n),\  \partial_{\psi}(T^{n}x)=-n+\partial_{\psi}(x).
\end{equation*}
Recall that a dynamical system $(\mathcal{X}, T)$ is said to satisfy the marker property if for every $N>0$
there exists an open set $U\subset \mathcal{X}$ satisfying
\begin{equation}\label{xiaohong}
U\cap T^{-n}U=\emptyset\ (1\leq n\leq N), \ \mathcal{X}=\bigcup_{n\in \mathbb{Z}}T^{-n}U.
\end{equation}
\begin{lem}\label{tiling}\cite{LT}
	Suppose $(\mathcal{X}, T)$ satisfies the marker property. Then for any $\epsilon > 0$ we can find a continuous function $\psi:\ \mathcal{X}\rightarrow [0,1]$ such that (\ref{xiaohua}) is nonempty for every $x\in \mathcal{X}$
	and that it satisfies that following two conditions.\\
	(1) There exists $M > 0$ such that $I_{\psi}(x,a)\subset (a-M, a+M)$ for all $x\in \mathcal{X}$ and $a\in \mathbb{Z}$.  The intervals $I_{\psi}(x,a)$ depend continuously on $x\in \mathcal{X}$, namely if $I_{\psi}(x,a)$
	has positive length and if $x_{k}\rightarrow x$ in $\mathcal{X}$ then $I_{\psi}(x_{k},a)$ converges to $I_{\psi}(x,a)$ in the Hausdorff topology.
	(2) The sets $\partial_{\psi}(x)$ are sufficiently "sparse" in the sense that
	\begin{equation}
	\lim_{R\rightarrow \infty}\frac{\sup_{x\in \mathcal{X}}|\partial_{\psi}(x)\cap [0, R]|}{R}< \epsilon.
	\end{equation}
	Here $|\partial_{\psi}(x)\cap [0, R]|$ is the cardinality of $\partial_{\psi}(x)\cap [0, R].$
\end{lem}
\subsection{Proof of Theorem \ref{le}}
Theorem \ref{aa} follows from following theorem. For a topological space $\mathcal{X}$ and a Banach space $(V, \|\cdot\|)$  we denote by $C(X, V)$ the  space of the continuous maps $f: \mathcal{X} \rightarrow V$ endowed the norm topology (i.e., the topology given by the metric $\sup_{x\in \mathcal{X}}\|f(x)-g(x)\|$). For convenience, we also give proof of Theorem \ref{le}.
\begin{theorem}\label{le}
	Let $(\mathcal{X}, T)$ be a dynamical system with a sub-additive potential $\mathcal{F}=\left\lbrace  \varphi_{n}\right\rbrace_{n=1}^{\infty} $, and let $(V, \left\| \cdot\right\| )$ be an infinite dimension Banach space. Suppose $(\mathcal{X}, T)$  has the marker property and  there exists $K>0$ such that $|\varphi_{n+1}-\varphi_{n}|<K$ for every $n$. Then for a dense subset $f\in C(\mathcal{X}, V)$, $f$ is a topological embedding and satisfies
	$${                                                                           \overline{{\rm mdim_{M}}}}(\mathcal{X}, T, f^{*}\left\| \cdot\right\|, \mathcal{F} )={\rm mdim}(\mathcal{X}, T, \mathcal{F}).$$
	Here $f^{*}\left\| \cdot \right\| $ is the metric $\left\| f(x)-f(y)\right\| ~(x, y\in \mathcal{X})$.
\end{theorem}
\begin{proof}
	First we introduce some notations. For a natural number $N$ we set $[N]=\left\lbrace 0, 1, 2, \cdots, N-1\right\rbrace $. We define a norm on $V^{N}$ (the $n$-th  power of $V$) by
	$$ \left\| (x_{0}, x_{1}, \cdots, x_{N-1})\right\|_{N}=\max\left\lbrace \left\|x_{0} \right\|, \left\|x_{1} \right\|, \cdots, \left\|x_{N-1} \right\|    \right\rbrace.  $$ For simplicial complexes  $P$ and $Q$ we define their join $P*Q$ as the quotient space of $[0,1]\times P \times Q$ by the equivalence relation
	$$ (0, p, q)\sim (0, p, q'),~~(1, p, q)\sim (1,p', q),~~(p, p' \in P, q,q'\in Q).$$
	We denote the equivalence class of $(t, p, q)$ by $(1-t)p\oplus tq$. We identify $P$ and $Q$ with $\left\lbrace (0, p, *)| p\in P\right\rbrace $ and $\left\lbrace (1, *. q) \right\rbrace $ in $P*Q$ respectively. For a continuous map $f: \mathcal{X} \rightarrow V$ and $I \subset \mathbb{R}$  we define $\Phi_{f, I}(x): \mathcal{X} \rightarrow V^{I\cap \mathbb{Z}}$ by
	$$\Phi_{f, I}(x)=(f(T^{a}x))_{a\in I\cap \mathbb{Z}}.$$
	For a natural number $R$ we set $\Phi_{f, R}:=\Phi_{f , [R]}: \mathcal{X}\rightarrow V^{R}$. We denote by $\Phi_{f, R}^{*}\left\| \cdot\right\|_{R} $ the semi-metric $\left\| \Phi_{f , [R]}(x)-\Phi_{f , [R]}(y) \right\| $ on $\mathcal{X}$. For a semi-metric $d'$ on $\mathcal{X}$ and $\epsilon>0$ we define
	\begin{align*}
	\#(\mathcal{X}, d', \varphi, \epsilon)=\inf  \{ \sum_{i=1}^{n} (1/\epsilon)^{\sup_{U_{i}}\varphi}\mid  ~ \mathcal{X}=& U_{1} \cup \cdots \cup U_{n} ~\text{ is an open cover with }  \\ &  {\rm diam} ~U_{i} < \epsilon ~\text{for all}  ~1\leq i \leq n \}.
	\end{align*}
	where ${\rm diam}(U_{i}, d') $ is the  supremum of $d'(x, y)$ over $x, y \in U_{i}$. We fix a continuous function $\alpha: \mathbb{R} \rightarrow [0,1]$ such that $\alpha(t)=1$ for $t\leq 1/2$ and $\alpha(t)=0$  for $t\geq 3/4.$
	
	We can assume $D={\rm mdim}(\mathcal{X}, T, \mathcal{F})<\infty.$ Fix a metric $d$ on $\mathcal{X}$. Take an arbitrary continuous map $f: \mathcal{X} \rightarrow V$ and $\eta>0$. Our purpose is to construct a topological embedding $f': \mathcal{X} \rightarrow V$ satisfying $\left\|  f(x) -f'(x)\right\|< \eta  $ and ${\overline{{\rm mdim_{M}}}}(\mathcal{X}, T, f^{'*}\left\|\cdot \right\|, \mathcal{F} )\leq D.$ We may assume that $f(\mathcal{X})$ is contained in the open unit ball $B_{1}^{\circ}(V)$. We will inductively construct the following data for $n\geq 1$.

		\begin{itemize}
			\item [(1)] $1/2>\epsilon_{1}>\epsilon_{2}>\cdots>0$ with $\epsilon_{n+1}<\epsilon_{n}/2$ and $\eta /2>\delta_{1}>\delta_{2}>\cdots>0$ with $\delta_{n+1}<\delta_{n}/2.$
			\item [(2)] A natural number $N_{n}.$
			\item [(3)] A continuous function $\psi_{n}: \mathcal{X} \rightarrow [0,1]$ such that for every $x \in \mathcal{X}$ there exists $a \in \mathbb{Z}$ satisfying $\psi_{n}(T^{a}x)>0$. We apply the dynamical tiling construction of subsection \ref{5.2} to $\psi_{n}$ and get the decomposition $\mathbb{R}=\bigcup\limits_{a\in \mathbb{Z}}I_{\psi_{n}}(x,a)$ for each $x \in \mathcal{X}$.
			\item [(4)] $(1/ n)$-embeddings $\pi_{n}: (\mathcal{X}, d_{N_{n}}) \rightarrow P_{n}$ and $\pi'_{n}: (\mathcal{X}, d) \rightarrow Q_{n}$ with simplicial complexes $P_{n}$ and $Q_{n}$.
			\item [(5)] For each $\lambda \in [N_{n}]$, a linear map $g_{n, \lambda}: P_{n} \rightarrow B_{1}^{\circ}(V).$
			\item [(6)] A linear map $g_{n}^{'}: Q_{n} \rightarrow B_{1}^{\circ}(V).$
		\end{itemize}
 We assume the following six conditions.
\begin{con}\label{con}
	\begin{enumerate}
		\item [(1)] For each $\lambda\in [N_{n}],$ the map $g_{n, \lambda}* g_{n}^{'}(P_{n}*Q_{n}): P_{n}* Q_{n}\rightarrow B_{1}^{\circ}(V)$ is injective. For $\lambda_{1}\neq \lambda_{2}$,
		$$g_{n, \lambda_{1}}* g_{n}^{'}(P_{n}*Q_{n})\cap g_{n, \lambda_{2}}* g_{n}^{'}(P_{n}*Q_{n})=g_{n}'(Q_{n}).$$
		\item [(2)] Set $g_{n}=(g_{n,0}, g_{n,1}, \cdots, g_{n, N_{n}-1}): P_{n}\rightarrow V^{N_{n}}$. We assume that $\pi_{n}$ is essential and
		$$\sum \limits_{\Delta \subset P_{n}}\left( \dfrac{1}{\epsilon}\right) ^{\sup_{\pi_{n}^{-1}(O_{P_{n}}(\Delta))}\varphi_{N}}  \#(g_{n}(\Delta), \|\cdot\|_{N_{n}}, \epsilon)<\left( \dfrac{1}{\epsilon}\right) ^{(D+\frac{3}{n}))N_{n}},~~~(0<\epsilon\leq \epsilon_{n}).$$
		Here $\Delta$ runs overs simplexes of $P_{n}$. Since $\pi_{n}$ is essential, $\pi_{n}^{-1}(O_{P_{n}}(\Delta))$ is non-empty for every $\Delta \subset P_{n}$.
		\item[(3)] For $0<\epsilon\leq \epsilon_{n-1}$ $(n\geq 2)$,
		$$\#(\mathcal{X}, (g_{n}\circ \pi_{n})^{*}\|\cdot\|_{N_{n}},  \varphi_{N}, \epsilon)<2^{N_{n}}\left( \dfrac{1}{\epsilon}\right) )^{(D+\frac{4}{n-1}))N_{n}}.$$
		Here $(g_{n}\circ\pi_{n})^{*}\|\cdot\|_{N_{n}}$ is the semi-metric $\| g_{n}(\pi_{n}(x))-g_{n}(\pi_{n}(y))\|$ on $\mathcal{X}$.
		\item[(4)] There exists $M_{n}>0$ such that $I_{\psi_{n}}(x, a)\subset (a-M_{n}, a+M_{n})$ for all $x\in \mathcal{X}$ and $a\in \mathbb{Z}$. We take $C_{n}\geq 1$ satisfying
		\begin{align}\label{cos}
		\#\left( \bigcup\limits_{\lambda\in [N_{n}]} g_{n,\lambda}*g_{n}^{'}(P_{n}*Q_{n},\|\cdot\|,\epsilon)\right) <\left( \dfrac{1}{\epsilon}\right)^{C_{n}}~~(0<\epsilon\leq \dfrac{1}{2}).
		\end{align}
		Then we assume
		$$\lim\limits_{R\rightarrow\infty}\dfrac{\sup_{x\in \mathcal{X}}|\partial_{\psi}(x)\cap [0,R]|}{R}<\dfrac{1}{2nN_{n}(C_{n}+\| \varphi_{1}\|_{\infty}+K)}.$$
		where $\| \varphi_{1}\|_{\infty}=\max_{\mathcal{X}}| \varphi_{1}(x)|.$
		\item[(5)] We define a continuous map $f_{n}: \mathcal{X}\rightarrow B_{1}^{\circ}(V)$ as follows. Let $x\in \mathcal{X}$. Take $a\in \mathbb{Z}$ with $0\in I_{\psi_{n}}(x, a)$, and take $b\in \mathbb{Z}$ satisfying $b\equiv a({\rm mod} N_{n})$ and $0\in b+N_{n}$. We set
		\begin{align}\label{5.6}
		f_{n}(x)=\left\lbrace 1-\alpha({\rm dist}(0, \partial_{\psi_{n}}(x)))\right\rbrace g_{n, -b}(\pi_{n}((T^{b}x)))+\alpha({\rm dist}(0, \partial_{\psi_{n}}(x))) g_{n}'(\pi_{n}'(x)),
		\end{align}
		where ${\rm dist}(0, \partial_{\psi_{n}}(x))=\min_{t\in \partial_{\psi_{n}(x)}}|t|.$ Then we assume that if a continuous map $f': \mathcal{X}\rightarrow V$ satisfies $\|f(x)- f'(x)\|<\delta_{n}$ for all $x\in \mathcal{X}$ then it is a $(1/n)$-embedding with respect to $d$.
	\end{enumerate}
\end{con}		
Suppose that  we have constructed the above data. We define a continuous map $f':\mathcal{X} \rightarrow V$ by $f'(x)=\lim\limits_{n\rightarrow \infty}f_{n}(x)$. It satisfies
$\|f'(x)-f(x)\|<\eta$ and $\|f'(x)-f_{n}(x)\|<\min(\epsilon_{n}/4, \delta_{n})$ for all $n\geq 1$. Then the condition  $(5)$  implies that $f'$ is a $(1/n)$-embedding with respect to $d$ for all $n\geq 1$, which means that $f'$ is a topological embedding. We estimate
\begin{align*}
{\rm mdim}_{M}(\mathcal{X}, T, (f')^{*}\|\cdot\|, \mathcal{F})=\limsup\limits_{ \epsilon \rightarrow 0}\left\lbrace \left(\lim\limits_{R\rightarrow \infty}\dfrac{\log \# (\mathcal{X}, \Phi_{f',R}^{*}\|\cdot\|_{R},  \varphi_{R}, \epsilon)}{R} \right)/\log (1/\epsilon) \right\rbrace.
\end{align*}
Let $0<\epsilon<\epsilon_{1}$. Take $n\geq 1$ with $\epsilon_{n}<\epsilon<\epsilon_{n-1}$. From $\|f'(x)-f_{n}(x)\|<\epsilon_{n}/4$,
$$\#(\mathcal{X}, \Phi_{f',R}^{*}\|\cdot\|_{R},  \varphi_{R}, \epsilon)\leq \#(\mathcal{X}, \Phi_{f_{n},R}^{*}\|\cdot\|_{R},  \varphi_{R}, \epsilon-\dfrac{\epsilon_{n}}{2})\leq \#(\mathcal{X}, \Phi_{f_{n},R}^{*}\|\cdot\|_{R},  \varphi_{R}, \dfrac{\epsilon}{2}).$$
From Claim \ref{num1} below,
$$ \lim\limits_{R\rightarrow \infty} \dfrac{\log \# (\mathcal{X}, \Phi_{f',R}^{*}\|\cdot\|_{R},  \varphi_{R}, \epsilon)}{R} \leq 2+ (D+\dfrac{4}{n-1}+\dfrac{1}{n})\log \left( \dfrac{2}{\epsilon}\right). $$
Since $n\rightarrow \infty$ as $\epsilon\rightarrow 0$, this proves $\overline{{\rm mdim_{M}}}(\mathcal{X}, T, (f')^{*}, \|\cdot\|, \mathcal{F})$.
\begin{claim}\label{num1}
	Let $0<\epsilon<\epsilon_{n-1} ~(n\geq 2)$. If $R$ is a sufficiently large natural number then
	$$\#(\mathcal{X}, \Phi_{f_{n},R}^{*}\|\cdot\|_{R},  \varphi_{R},\epsilon)\leq 4^{R}\left( \dfrac{1}{\epsilon}\right)^{(D+\frac{4}{n-1})R+\frac{R}{n}} $$
\end{claim}
\begin{proof}
	Let $x\in \mathcal{X}$. A discrete interval $J=[b, b+N_{n})\cap \mathbb{Z}$ of length $N_{n}$ $(b \in \mathbb{Z})$ is said to be ${\bf good~ for}$ $x$ if there exists $a\in\mathbb{Z}$ such that $b\equiv a ({\rm mod} N_{n})$ and $[b-1, b+N_{n}]\subset I_{\Psi_{n}}(x, a)$. If $J$ is good for  $x$ then
	$$\Phi_{f_{n},J}(x)=g_{n}(\pi_{n}(T^{b}x))\in g_{n}(P_{n}).$$
	
	We denote by $\mathcal{J}_{x}$ the union of $J\subset [R]$ which are good for $x$. For a subset $\mathcal{J} \subset [R]$ we define $\mathcal{X}_{\mathcal{J}}$ as the set of $x\in \mathcal{X}$ satisfying $\mathcal{J}_{x}=\mathcal{J}.$ The set $\mathcal{X}_{\mathcal{J}}$ may be empty. If it is non-empty, then from Condition $\ref{con}$ (3)
	\begin{align}\label{num}
	\#(\mathcal{X}_{\mathcal{J}}, \Phi_{f_{n},R}^{*}\|\cdot\|_{R},  \varphi_{R}, \epsilon )\leq \left\lbrace 2^{N_{n}}\left( \dfrac{1}{\epsilon}\right) ^{(D+\frac{4}{n-1})N_{n}}\right\rbrace^{|\mathcal{J}|/N_{n}} \cdot \left( \dfrac{1}{\epsilon}\right)^{(C_{n}+ \| \varphi_{1}\|_{\infty}+K)|[R]\setminus \mathcal{J}|}
	\end{align}	
	Here $C_{n}$ is the positive constant introduced in (\ref{cos}). We have $|\mathcal{J}|\leq R$ and $$|[R]\setminus \mathcal{J}|\leq 2N_{n}\sup\limits_{x\in \mathcal{X}}|\partial_{\Psi_{n}(x)}\cap [0,R]|+2N_{n}.$$
	The second term $"+2N_{n}"$ in the right-hand side is the edge effect. From Condition \ref{con} (4), for sufficiently larger $R$
	$$(C_{n}+\| \varphi_{1}\|_{\infty}+K)|[R]\setminus \mathcal{J}|<\dfrac{R}{n}.$$
	Then the quantity (\ref {num}) is bounded by
	$$ 2^{R}\left( \dfrac{1}{\epsilon}\right)^{(D+\frac{4}{n-1}))R+\frac{R}{n}}.$$
	The number of the choices of $\mathcal{J} \subset [R]$ is bounded by $2^{R}$. Thus
	$$\#(\mathcal{X}, \Phi_{f_{n},R}^{*}\|\cdot\|_{R}, \varphi_{R}, \epsilon ) \leq 4^{R}\left( \dfrac{1}{\epsilon}\right)^{(D+\frac{4}{n-1})R+\frac{R}{n}}.$$
\end{proof}
{\bf Induction: Step 1.} Now we start to construct the data. First we construct them for $n=1$. By the continuity of $f$ and ${\rm mdim}(\mathcal{X}, T, \mathcal{F})=D$. Take small enough  $0<\tau_{1}<1$, there exists $N_{1}>0$, a simplicial complex $P_{1}$ and a $\pi_{1}$-embedding $\pi_{1}: (\mathcal{X}, d_{N_{1}})\rightarrow P_{1}$ such that
\begin{itemize}
	\item
	$ d(x, y)<\tau_{1}\Rightarrow ~\|f(x)-f(y)\|<\dfrac{\eta}{2}.$
	\item ${\rm dim}_{\pi_{1}(x)}P_{1}+ \varphi_{N_{1}}(x)<N_{1}(D+1)$ for all $x\in \mathcal{X}$.
	\item $\dfrac{{\rm var}_{\tau_{1}}( \varphi_{N_{1}}, d_{N_{1}})}{N_{1}}< 1,$
	where ${\rm var}_{\epsilon}(\varphi, d)=\sup \left\lbrace |\varphi(x)-\varphi(y)|, ~d(x, y)<\epsilon \right\rbrace $.
\end{itemize}

We also take a simplicial complex $Q_{1}$ and a $\tau_{1}$-embedding $\pi_{1}^{'}: (\mathcal{X},d) \rightarrow Q_{1}$. By subdividing $P_{1}$ and $Q_{1}$  if necessary, we can assume that all simplexes $\Delta \subset P_{1}$ and all $\omega \in {\rm Ver(Q_{1})}$
$${\rm diam}(\pi_{1}^{-1}(O_{P_{1}}(\Delta)), d_{N_{1}}))<\tau_{1},~~{\rm diam}((\pi_{1}^{'})^{-1}(O_{Q_{1}}(\omega)), d))<\tau_{1}.$$
Moreover by Lemma \ref{lem4} we can assume that $\pi_{1}$ is essential.
By Lemma \ref{key} (3) there exist linear maps $g_{1, \lambda}: P_{1} \rightarrow B_{1}^{\circ}(V)$ $(\lambda \in [N_{1}])$ and $g_{1}^{'}: Q_{1} \rightarrow B_{1}^{\circ}(V)$
satisfying
\begin{align}\label{5.7}
\|f(T^{\lambda}x)-g_{1, \lambda}(\pi_{1}(x))\|<\dfrac{\eta}{2},~~\| f(x)-g_{1}^{'}(\pi_{1}(x))\|<\dfrac{\eta}{2}.
\end{align}
We slightly perturb $g_{1, \lambda}$ and $g_{1}^{'}$ (if necessary ) by Lemma \ref{key} (2) so that they satisfy Condition \ref{con} (1).
By Lemma \ref{key} (1), we can choose $0<\epsilon_{1}<1/2$ such that for any $0<\epsilon\leq \epsilon_{1}$ and simplex $\Delta \subset P_{1}$
$$\#(g_{1}(\Delta), \|\cdot\|_{N_{1}}, \epsilon))<\dfrac{1}{(\text{Number of simplexes of}~ P_{1})}\left( \dfrac{1}{\epsilon} \right)^{{\rm dim}\Delta+1}. $$
Let $\Delta \subset P_{1}$ be a simplex. Since $\pi_{1}$ is essential, we can find a point $x\in \pi_{1}^{-1}(O_{P_{1}}(\Delta))$ with ${\rm dim}(\Delta)\leq  {\rm dim}_{\pi_{1}(x)}P_{1}.$ From the choice of $\tau_{1}$
$$ \sup\limits_{\pi_{1}^{-1}(O_{P_{1}}(\Delta))}  \varphi_{N_{1}}\leq  \varphi_{N_{1}}(x)+ N_{1}.$$
Hence for $0<\epsilon\leq \epsilon_{1}$
\begin{align*}
&\left( \dfrac{1}{\epsilon} \right)^{\sup_{\pi_{1}^{-1}(O_{P_{1}}(\Delta))} \varphi_{N_{1}}} \#(g_{1}(\Delta), \|\cdot\|_{N_{1}}, \epsilon)\\& <\dfrac{1}{(\text{Number of simplexes of }P_{1}) }\left( \dfrac{1}{\epsilon}\right)^{{\rm dim}(\Delta)+ \varphi_{N_{1}}+ N_{1}+1}\\&\leq \dfrac{1}{(\text{Number of simplexes of } P_{1})}\left( \dfrac{1}{\epsilon}\right)^{{\rm dim}_{\pi_{1}(x)}P_{1}+ \varphi_{N_{1}}+ N_{1}+1}
\end{align*}
From ${\rm dim}_{\pi_{1}(x)}P_{1}+ \varphi_{N_{1}}(x)<N_{1}(D+1)$, this is bounded by
\begin{align*}
&\dfrac{1}{(\text{Number of simplexes of } P_{1})}\left( \dfrac{1}{\epsilon}\right)^{N_{1}(D+1)+ N_{1}+1}\\&\leq \dfrac{1}{(\text{Number of simplexes of } P_{1})}\left( \dfrac{1}{\epsilon}\right)^{N_{1}(D+3)}
\end{align*}
This shows Condition \ref{con} (2):
$$ \sum\limits_{\Delta \subset P_{1}} \left( \dfrac{1}{\epsilon}\right)^{\sup_{\pi_{1}^{-1}(O_{P_{1}}(\Delta))}}\#(g_{1}(\Delta, \|\cdot\|_{N_{1}}, \epsilon))<\left( \dfrac{1}{\epsilon} \right)^{N_{1}(D+3)}.  $$
Condition \ref{con} (3) is empty for $n=1$.
By Lemma \ref{tiling} we can choose a continuous function $\Psi_{1}: \mathcal{X} \rightarrow [0,1]$ satisfying Condition \ref{con} (4). The continuous map $f_{1}: \mathcal{X} \rightarrow V$ defined in (\ref{5.6}) is a 1-embedding. Since "1-embedding" is  an open condition, we can choose $0<\delta_{1}<\eta/2$ such that any continuous map $f': \mathcal{X} \rightarrow V$ with $\|f'(x)- f_{1}(x)\|<\delta_{1}$ is also a 1-embedding. This establishes Condition \ref{con} (5).
From (\ref{5.7}) we get Condition \ref{con} (6): $$\|f(x)-f_{1}(x)\|<\eta/2.$$ We have completed the construction of the data for $n=1$.

{\bf Induction: Step n $ \rightarrow$ Step n+1}

Suppose we have  constructed the data for $n$. We will construct the data for $n+1$.
We subdivide the join $P_{n}*Q_{n}$ sufficiently fine (denote by $\overline{P_{n}*Q_{n}}$) such that for all simplexes $\Delta \subset \overline{P_{n}*Q_{n}}$ and all $\lambda\in [N_{n}]$
\begin{align}\label{5.9}
{\rm diam}(g_{n, \lambda}*g_{n}^{'}(\Delta), \|\cdot\|)<\min\left( \dfrac{\epsilon_{n}}{8}, \dfrac{\delta_{n}}{2}\right).
\end{align}
We define a continuous map $q_{n}:\mathcal{X} \rightarrow \overline{P_{n}*Q_{n}}$ as follows. Let $x\in \mathcal{X}$. Take $a, b \in \mathbb{Z}$ such that $0\in I_{\Psi_{n}
}(x, a)$, $a\equiv b ({\rm mod} N_{n})$ and $0\in b+[N_{n}]$. Then we set
$$q_{n}(x)=\left\lbrace 1-\alpha({\rm dist}(0, \partial_{\psi_{n}}(x))\right\rbrace\pi_{n}(T^{b}x)\oplus\alpha({\rm dist}(0, \partial_{\psi_{n}}(x))) \pi_{n}^{'}(x). $$
We have
\begin{align}\label{5.10}
f_{n}(x)=g_{n, -b}*g_{n}^{'}(q_{n}(x)).
\end{align}

Take $0<\tau_{n+1}<1/n+1$ satisfying the following four conditions.
\begin{itemize}
	\item [(i)] If $d(x,y)<\tau_{n+1}$ then $\| f_{n}(x)-f_{n}(y)\|<\min(\epsilon_{n}/8, \delta_{n}/2)$.
	\item [(ii)] If $d(x,y)<\tau_{n+1}$ then then the  decompositions of dynamical tiling  are  "close" in the following two senses.
	\begin{itemize}
		\item [$\cdot$]$|{\rm dist}(0, \partial_{\psi_{n}}(x))-{\rm dist}(0, \partial_{\psi_{n}}(y))|<\dfrac{1}{4}$
		\item [$\cdot$] If $(-1/4, 1/4)\subset I_{\Psi_{n}}(x, a)$ then $0$ is an interior point of $I_{\Psi_{n}}(y, a)$.
	\end{itemize}
	\item[(iv)] Consider the open cover $\left\lbrace q_{n}^{-1} (\overline{P_{n}*Q_{n}}(v))\right\rbrace_{v\in {\rm Ver}(\overline{P_{n}*Q_{n}})} $ of $\mathcal{X}$. The number $\tau_{n+1}$ is smaller than its Lebesgue number:
	$$ \tau_{n+1}<LN\left( \mathcal{X}, d, \left\lbrace q_{n}^{-1} (\overline{P_{n}*Q_{n}}(v))\right\rbrace_{v\in {\rm Ver}(\overline{P_{n}*Q_{n}})}\right). $$
\end{itemize}
Take a $\tau_{n+1}$-embedding $\pi_{n+1}^{'}: (\mathcal{X}, d) \rightarrow Q_{n+1}$ with a simplicial complex $Q_{n+1}$. By subdividing it, we can assume that ${\rm diam}((\pi_{1}^{'})^{-1}(O_{Q_{n+1}}(\omega), d))<\tau_{n+1}$ for all $\omega \in {\rm Ver}(Q_{n+1})$. By Lemma \ref{key} (3) there exists a linear map $\tilde{g}_{n+1}^{'}: Q_{n+1} \rightarrow B_{1}^{\circ}(V)$ satisfying
\begin{align}\label{per2}
\| f_{n}(x)-\tilde{g}_{n+1}^{'}(\pi_{n+1}(x))\|<\min\left(  \dfrac{\epsilon_{n}}{8}, \dfrac{\delta_{n}}{2}\right).
\end{align}
Take $N_{n+1}\geq N_{n}$ satisfying two conditions.
\begin{itemize}
	\item [(a)]
	There exists a $\tau_{n+1}$-embedding $\pi_{n+1}: (\mathcal{X}, d_{N_{n+1}}) \rightarrow P_{n+1}$ with a simplex complex $P_{n+1}$ such that for all $x\in \mathcal{X}$
	\begin{align}\label{cc}
	\dfrac{{\rm dim}_{\pi_{n+1}(x)}P_{n+1}+\varphi_{N_{n+1}}(x)}{N_{n+1}}< D+\dfrac{1}{n+1}.
	\end{align}
	\item [(b)]
	$$\dfrac{1+\sup\limits_{x\in \mathcal{X}}|\partial_{\psi_{n}}(x)\cap [0, N_{n+1}]|}{N_{n+1}}<\dfrac{1}{2nN_{N}(C_{n}+\| \varphi_{1}\|_{\infty}+K)},$$
	where $C_{n}$ is the positive constant introduced in (\ref{cos}).
	\item[(c)] $$\dfrac{{\rm var}_{\tau_{n+1}}( \varphi_{N_{n+1}}, d_{N_{1}})}{N_{n+1}}< \dfrac{1}{n+1}.$$
\end{itemize}

By subdividing $P_{n+1}$  if necessary, we can assume that for any simplexes $\Delta, \Delta'\subset P_{n+1}$ with $\Delta \cap \Delta' \neq \emptyset$
\begin{align}\label{pi}
{\rm diam}(\pi_{n+1}^{-1}(O_{P_{n+1}}(\Delta)\cup\pi_{n+1}^{-1}(O_{P_{n+1}}(\Delta'), d_{N_{n+1}}))<\tau_{n+1}.
\end{align}
Moreover by Lemma \ref{lem4} we can assume that $\pi_{n+1}$ is essential.

By the choice of $\tau_{n+1}$, we apply  Lemma \ref{lem5.3} to $\pi_{n+1}: \mathcal{X} \rightarrow P_{n+1}$ and $q_{n}\circ T^{\lambda}: \mathcal{X} \rightarrow \overline{P_{n}*Q_{n}} (\lambda \in [N_{n+1}]).$ Then we can get simplicial maps $h_{\lambda}: P_{n+1} \rightarrow \overline{P_{n}*Q_{n}} (\lambda \in [N_{n+1}])$ satisfying the  three condition:
\begin{itemize}
	\item [(A)] For every $\lambda \in [N_{n+1}]$ and $x\in \mathcal{X}$, the two points $h_{\lambda}(\pi_{n+1}(x)) $ and $q_{n}(T^{\lambda}x)$ belong to the same simplex of $\overline{P_{n}*Q_{n}}.$
	\item [(B)] Let $\lambda \in [N_{n+1}]$ and $\Delta \subset P_{n+1}$ be a simplex. If $\pi_{n+1}^{-1}(O_{P_{n+1}}(\Delta)\subset T^{-\lambda}q_{n}^{-1}(\overline{P_{n}}),$
	then $h_{\lambda}(\Delta)\subset \overline{P_{n}}.$ Similarly, if $\pi_{n+1}^{-1}(O_{P_{n+1}}(\Delta)\subset T^{-\lambda}q_{n}^{-1}(\overline{Q_{n}}),$, then $h_{\lambda}(\Delta)\subset \overline{Q_{n}}.$		
	\item [(C)] Let $\lambda, \lambda' \in [N_{n+1}]$ and $\Delta \subset P_{n+1}$ be a simplex. If $q_{n}\circ T^{\lambda}=q_{n}\circ T^{\lambda'}$ on $\pi_{n+1}^{-1}(O_{P_{n+1}}(\Delta)$ then $h_{\lambda}=h_{\lambda'}$ on $\Delta$.
\end{itemize}
Define a linear map $\tilde{g}_{n+1,\lambda}: P_{n+1} \rightarrow B_{1}^{\circ}(V)$ for each $\lambda \in [N_{n+1}]$ as follows. For each $\Delta \in P_{n+1}$, since $\pi_{n+1}$ is essential, we can find a point $x\in \pi_{n+1}^{-1}(O_{P_{n+1}}(\Delta))$. Take $a, b \in \mathbb{Z}$ such that $\lambda\in I_{\Psi_{n}
}(x, a)$, $a\equiv b ({\rm mod} N_{n})$ and $\lambda\in b+[N_{n}]$.

Set $$\tilde{g}_{n+1,\lambda}(u)=g_{n, \lambda-b}*g_{n}^{'}(h_{\lambda}(u)) ~~(u \in \Delta).$$

From (\ref{5.9}) and  (\ref{5.10}),
\begin{align}\label{per1}
\|\tilde{g}_{n+1,\lambda}(\pi_{n+1}(x))-f_{n}(T^{\lambda}x)\| <\min\left( \dfrac{\epsilon_{n}}{8}, \dfrac{\delta_{n}}{2}\right).
\end{align}
\begin{claim}
	The above construction of $\tilde{g}_{n+1,\lambda}$ is independent of the various choices.
\end{claim}
\begin{proof}
	see \cite{LT}.
\end{proof}
\begin{claim}\label{num2}
	Set $\tilde{g}_{n+1}=(\tilde{g}_{n+1,0},\cdots, \tilde{g}_{n+1, N_{n+1}-1}): P_{n+1} \rightarrow V^{N_{n+1}}.$ For  $0< \epsilon\leq \epsilon_{n}$
	$$\#(\mathcal{X}, (\tilde{g}_{n+1}\circ \pi_{n+1})^{*}\|\cdot\|_{N_{n+1}}, \varphi_{N_{n+1}}, \epsilon)< 2^{N_{n+1}}\left( \dfrac{1}{\epsilon} \right)^{(D+\frac{4}{n})N_{n+1}}. $$
\end{claim}
\begin{proof}
	This is close to the proof of Claim \ref{num1}. But it is  a bit more involved. Let $x\in \mathcal{X}$. We say that a discrete interval $J=[b, b+N_{n})\cap \mathbb{Z}$ of length $N_{n} (b \in \mathbb{Z})$ is good for $x$ if $J\subset [N_{n+1}]$ and there exists $a\in \mathbb{Z}$ satisfying $b\eqcirc a ({\rm mod} N_{n} )$ and $[b-1, b+N_{n}]\subset I_{\psi_{n}}(x,a)$.
	
	Suppose $J=[b, b+N_{n}) \cap \mathbb{Z}$ is good for $x\in \mathcal{X}$. Take a simplex $\Delta\subset P_{n+1}$ containing $\pi_{n+1}(x).$ Let $y\in \pi_{n+1}^{-1}(O_{P_{n+1}}(\Delta))$ be an arbitrary point. From $(\ref{pi})$ we have $d_{N_{n+1}}(x, y)<\tau_{n+1}$. From the condition (iii) of the choice of $\tau_{n+1}$,
	$$[b-\dfrac{3}{4}, b+N_{n}-\dfrac{1}{4}]\subset I_{\psi_{n}}(y ,a).$$
	Then for all $\lambda \in J$
	$$q_{n}(T^{\lambda}y)=q_{n}(T^{b}y)=\pi_{n}(T^{b}y)\in \overline{P_{n}}.$$
	From the condition $(B)$ and $(C)$ of the choice of $h_{\lambda}$,
	$$h_{b}(\Delta)\subset \overline{P_{n}},~~h_{\lambda}=h_{b} ~\text{on}~\Delta ~\text{for }~\lambda\in J.$$
	Then $$(\tilde{g}_{n+1, \lambda}(\pi_{n+1}(x)))_{\lambda\in J}=g_{n}(h_{b}(\pi_{n+1}(x))).$$
	Moreover it follows from the condition $(A)$ of the choice of $h_{\lambda}$ that $h_{b}(\pi_{n+1}(x))$ and $q_{n}(T^{b}x)=\pi_{n}(T^{b}x)$ belongs to the same simplex of $\overline{P_{n}}$.
	For $x\in \mathcal{X}$ we denote by $\mathcal{J}_{x}$ the union of the intervals $J\subset [N_{n+1}]$ good for $x$. For a subset $\mathcal{J}\subset [N_{n+1}]$ we define $\mathcal{X}_{\mathcal{J}}$ as the set of $x\in \mathcal{X}$ with $\mathcal{J}_{x}=\mathcal{J}$. The set $\mathcal{X}_{\mathcal{J}}$ may be empty. If it is non-empty, then from Condition \ref{con} (2)
	\begin{align}\label{11}
	&\#(\mathcal{X}_{\mathcal{J}}, (\tilde{g}_{n+1}\circ \pi_{n+1})^{*}\|\cdot\|_{N_{n+1}},  \varphi_{N_{n+1}}, \epsilon)\\& \nonumber <\left\lbrace \left( \dfrac{1}{\epsilon}\right) ^{(D+\frac{3}{n}))N_{n}} \right\rbrace^{|\mathcal{J}|/N_{n}}\cdot \left\lbrace \left( \dfrac{1}{\epsilon}\right) ^{C_{n}+\|\varphi_{1}\|_{\infty}+K}\right\rbrace^{|[N_{n+1}]\setminus \mathcal{J}|}.
	\end{align}
	We have $|\mathcal{J}|\leq N_{n+1}$ and
	\begin{align*}
	|[N_{n+1}]\setminus \mathcal{J}||&\leq  2N_{n}|\partial_{\psi_{n}}(x)\cap[0, N_{n+1}]|+2N_{n}\\&<\dfrac{N_{n+1}}{n(C_{n}+\|\varphi_{1}\|_{\infty}+K)} ~\text{by the condition  (b) of the choice of } N_{n+1}.
	\end{align*}
	Then the above $(\ref{11})$ is bounded by
	$$\left( \dfrac{1}{\epsilon}\right)^{(D+\frac{3}{n})N_{n+1}+\frac{N_{n+1}}{n}}= \left( \dfrac{1}{\epsilon}\right)^{(D+\frac{4}{n})N_{n+1}}.$$
	The number of the choices of $\mathcal{J}\subset [N_{n+1}]$ is bounded by $2^{N_{n+1}}$. Thus
	$$\#(\mathcal{X},(\tilde{g}_{n+1}\circ \pi_{n+1})^{*}\|\cdot\|_{N_{n+1}}, \varphi_{N_{n+1}}, \epsilon)<2^{N_{n+1}}\left( \dfrac{1}{\epsilon}\right)^{(D+\frac{4}{n})N_{n+1}}.$$
\end{proof}
From Lemma \ref{key} (1), we can  take $0<\epsilon_{n+1}<\epsilon_{n}/2$ such that for any $0<\epsilon\leq \epsilon_{n+1}$ and any linear map $g: P_{n+1} \rightarrow V^{N_{n}+1}$ with $g(P_{n+1}) \subset B_{1}^{\circ}(V)^{N_{n+1}}$
\begin{align*}
\#(g(\Delta), \|\cdot\|_{N_{n+1}}, \epsilon)<\dfrac{1}{\left( \text{Number of simplexes of } P_{n+1}\right) } \left( \dfrac{1}{\epsilon}\right)^{{\rm dim}(\Delta)+\frac{1}{n+1}}
\end{align*}
for all simplexes $\Delta \subset P_{n+1}$.

Let $g: P_{n+1}\rightarrow B_{1}^{\circ}(V)^{N_{n+1}}$ be a linear map and let $\Delta \subset P_{n+1}$ be  a simplex. Since $\pi_{n+1}$ is essential, we can find a point $x\in \pi_{n+1}^{-1}(O_{P_{n+1}})$ with  ${\rm dim}_{\pi_{n+1}(x)}P_{n+1}\geq {\rm dim}(\Delta).$ From (\ref{pi}) and the condition (ii) of the choice of $\tau_{n+1}$
$$\sup\limits_{\pi_{n+1}^{-1}(O_{P_{n+1}}(\Delta))}  \varphi_{N_{n+1}} \leq \varphi_{N_{n+1}}(x)+\dfrac{N_{n+1}}{n+1}.$$
Then for $0<\epsilon\leq \epsilon_{n+1}$
\begin{align*}
&\left( \dfrac{1}{\epsilon}\right)^{\sup\limits_{\pi_{n+1}^{-1}(O_{P_{n+1}}(\Delta))}\varphi_{N_{n+1}}}\#(g(\Delta), \|\cdot\|_{N_{n+1}}, \epsilon)\\&<\dfrac{1}{\left( \text{Number of simplexes of } P_{n+1}\right) } \left( \dfrac{1}{\epsilon}\right)^{\varphi_{N_{n+1}}(x)+{\rm dim}(\Delta)+\frac{N_{n+1}}{n+1}+\frac{1}{n+1}}\\&\leq \dfrac{1}{\left( \text{Number of simplexes of } P_{n+1}\right) } \left( \dfrac{1}{\epsilon}\right)^{\varphi_{N_{n+1}}(x)+{\rm dim}_{\pi_{n+1}(x)}P_{n+1}+\frac{N_{n+1}+1}{n+1}}\\&\leq\dfrac{1}{\left( \text{Number of simplexes of } P_{n+1}\right) } \left( \dfrac{1}{\epsilon}\right)^{(D+\frac{1}{n+1})N_{n+1}+\frac{N_{n+1}+1}{n+1}}~~~\text{by} ~(\ref{cc})\\&\leq \dfrac{1}{\left( \text{Number of simplexes of } P_{n+1}\right)}\left(\dfrac{1}{\epsilon} \right)^{(D+\frac{3}{n+1})N_{n+1}}.
\end{align*}
Hence for any $0<\epsilon<\epsilon_{n+1}$ and any linear map $g: P_{n+1}\rightarrow B_{1}^{\circ}(V)^{N_{n+1}}$
\begin{align}\label{w}
\sum\limits_{\Delta \subset P_{n+1}}\left( \dfrac{1}{\epsilon}\right)^{\sup\limits_{\pi_{n+1}^{-1}(O_{P_{n+1}}(\Delta))}\varphi_{N_{n+1}} }\#(g(\Delta), \|\cdot\|_{N_{n+1}}, \epsilon)<\left(\dfrac{1}{\epsilon} \right)^{(D+\frac{3}{n+1})N_{n+1}}.
\end{align}
We define $g'_{n+1}: Q_{n+1} \rightarrow B_{1}^{\circ}(V)$ and $g_{n+1, \lambda}: P_{n+1} \rightarrow B_{1}^{\circ}(V)~(\lambda \in [N_{n+1}])$ as small perturbations of $\tilde{g}'_{n+1}$ and $\tilde{g}_{n+1, \lambda}$ respectively. By Lemma \ref{key} (2), we can assume that they satisfy Condition \ref{con} (1). From $(\ref{per2})$  and $(\ref{per1})$ we can assume that the perturbations are so small that they satisfy
\begin{align}\label{f}
&\|g_{n+1}'(\pi_{n+1}(x)- f_{n}(x))\|<\min (\dfrac{\epsilon_{n}}{8}, \dfrac{\delta_{n}}{2}), \\&
\|g_{n+1, \lambda}(\pi_{n+1}(x))-f_{n}(T^{\lambda}x)\|<\min(\dfrac{\epsilon_{n}}{8}, \dfrac{\delta_{n}}{2}).
\end{align}
Moreover, from Claim \ref{num2}, we can assume that $g_{n+1}:=(g_{n+1, 0}, \cdots, g_{n+1, N_{n+1}-1})$ satisfies
$$\#(\mathcal{X}, ({g}_{n+1}\circ \pi_{n+1})^{*}\|\cdot\|_{N_{n+1}},  \varphi_{N_{n+1}}, \epsilon)< 2^{N_{n+1}}\left( \dfrac{1}{\epsilon} \right)^{(D+\frac{4}{n})N_{n+1}}$$ for all $\epsilon_{n+1}\leq \epsilon \leq \epsilon_{n}.$ On the other hand, from (\ref{w}), for $0<\epsilon\leq\epsilon_{n+1}$
$$\sum\limits_{\Delta \subset P_{n+1}}\left( \dfrac{1}{\epsilon}\right)^{\sup\limits_{\pi_{n+1}^{-1}(O_{P_{n+1}}(\Delta))}\varphi_{N_{n+1}} }\#(g_{n+1}(\Delta), \|\cdot\|_{N_{n+1}}, \epsilon)<\left(\dfrac{1}{\epsilon} \right)^{(D+\frac{3}{n+1})N_{n+1}}.$$
Thus we have established Condition \ref{con} (2) and (3) for $(n+1)$-th step. From Lemma \ref{tiling}, we can take a continuous function $\psi_{n+1}: \mathcal{X}\rightarrow [0,1]$ satisfying Condition \ref{con} (4). The map $f_{n+1}$ defined by $(\ref{5.6})$ is a $(1/n)$-embedding with respect to $d$ by Condition \ref{con} (1). Since $(1/n)$-embedding is an open condition, we can take $\delta_{n+1}>0$ satisfying Condition \ref{con} (5). From $(\ref{f})$,
\begin{align}
\|f_{n+1}(x)- f_{n}(x)\|<\min\left( \dfrac{\epsilon_{n}}{8}, \dfrac{\delta_{n}}{2} \right).
\end{align}
This shows Condition \ref{con} (6). We have finished the constructed of all data for the $(n+1)$-th step.
\end{proof}

{\bf Acknowledgements.}
The first and second author were supported by NNSF of China (11671208 and 11431012). We would like to express our gratitude to Tianyuan Mathematical Center in Southwest China, Sichuan University and Southwest Jiaotong University for their support and hospitality.

\end{document}